\documentclass[11pt, reqno]{amsart}

\usepackage{graphicx}
\usepackage{times}

\usepackage{xcolor}
\usepackage{hyperref}
\hypersetup{
    linktoc=all,
    colorlinks=true
}
\usepackage{stmaryrd}
\usepackage{xcolor}
\usepackage{hyperref}
\hypersetup{
    linktoc=all,
    colorlinks=true
}
\usepackage{stmaryrd}
\hypersetup{
    colorlinks = true,
    linkcolor = blue,
    anchorcolor = blue,
    citecolor = green,
    filecolor = blue,
    urlcolor = blue
    }

\newtheorem{theorem}{Theorem}[section]

\newtheorem{lemma}[theorem]{Lemma}

\newtheorem{ass}[theorem]{Assumption}

\theoremstyle{definition}
\newtheorem{definition}[theorem]{Definition}

\theoremstyle{remark}
\newtheorem{remark}[theorem]{Remark}
\numberwithin{equation}{section}

\usepackage[top=0.5in, bottom=0.75in, left=0.75in, right=0.75in]{geometry}

\usepackage[scr]{rsfso}
\usepackage[english]{babel}
\usepackage{fancyhdr}
\usepackage{amsmath}
\usepackage{amsthm}
\usepackage{amssymb}
\usepackage{eucal}
\usepackage{comment}
\usepackage{enumitem}
\setlist{leftmargin=*}
\usepackage[integrals]{wasysym}
\newcommand\nc{\newcommand}
\nc{\on}{\operatorname}
\nc{\E}{\mathbf{E}}
\nc{\R}{\mathbb R}
\nc{\C}{\mathbb C}
\nc{\Q}{\mathbb Q}
\nc{\Z}{\mathbb Z}
\nc{\N}{\mathbb N}
\nc{\F}{\mathbb F}
\nc{\wt}{\widetilde}
\nc{\ol}{\overline}
\nc{\short}[3]{0 \longrightarrow #1 \longrightarrow #2 \longrightarrow #3 \longrightarrow 0}
\nc{\pd}[2]{\frac{\partial #1}{\partial #2}}
\nc{\rnc}{\renewcommand}
\nc{\e}{\varepsilon}
\nc{\DMO}{\DeclareMathOperator}
\nc{\grad}{\nabla}
\nc{\Exp}{\mathbf{Exp}}
\nc{\fsp}{\fontdimen2\font=2.2pt}

\rnc{\t}{\mathfrak{t}}
\nc{\s}{\mathfrak{s}}
\nc{\x}{\mathrm{x}}
\nc{\y}{\mathrm{y}}
\nc{\z}{\mathrm{z}}
\rnc{\leq}{\leqslant}
\rnc{\geq}{\geqslant}
\rnc{\d}{\mathrm{d}}
\rnc{\O}{\mathrm{O}}
\newenvironment{nouppercase}{%
  \renewcommand{\uppercasenonmath}[1]{}}{}
\title{\large Hairer-Quastel universality in non-stationarity via energy solution theory}
\author{Kevin Yang}
\usepackage{setspace}
\begin{document}
\setstretch{1.0}
\fontdimen2\font=2.2pt
\raggedbottom
\begin{nouppercase}
\maketitle
\end{nouppercase}
\vspace{-25pt}
\begin{abstract}
The paper addresses \emph{probabilistic} aspects of the KPZ equation and stochastic Burgers equation by providing a solution theory that builds on the energy solution theory \cite{GJ15,GJara,GP,GP20}. The perspective we adopt is to study the stochastic Burgers equation by writing its solution as a probabilistic solution \cite{GP17} plus a term that can be studied with deterministic PDE considerations. One motivation is universality of KPZ and stochastic Burgers equations for a certain class of stochastic PDE growth models, first studied in \cite{HQ}. For this, we prove universality for SPDEs with general nonlinearities, thereby extending \cite{HQ,HX}, and for many non-stationary initial data, thereby extending \cite{GP16}. 

Our perspective lets us also prove explicit rates of convergence to white noise invariant measure of stochastic Burgers for \emph{non-stationary} initial data, in particular extending the spectral gap result of \cite{GP20} beyond stationary initial data, though for non-stationary data our convergence will be measured in Wasserstein distance and relative entropy, not via the spectral gap as in \cite{GP20}. Actually, we extend the spectral gap in \cite{GP20} to a log-Sobolev inequality. Our methods can also analyze fractional stochastic Burgers equations \cite{GP20}; we discuss this briefly. 

Lastly, we note our perspective on the KPZ and stochastic Burgers equations gives a first \emph{intrinsic} notion of solutions for general continuous initial data, in contrast to H\"{o}lder regular data needed for regularity structures \cite{Hai14}, paracontrolled distributions \cite{GPI}, and H\"{o}lder-regular Brownian bridge data for energy solutions \cite{GJara,GP}.
\end{abstract}
\section{Introduction}
The \emph{Kardar-Parisi-Zhang} (KPZ) equation is a conjecturally canonical model for interface fluctuations derived in \cite{KPZ} via non-rigorous renormalization group heuristics. Such ``interface fluctuations" include flame propagation, infected individuals in an epidemic, bacterial colony and tumor growth, and crack formations; see \cite{C11} and \cite{KPZ}. However, rigorous proof of this so-called ``weak KPZ universality" problem has posed a challenging mathematical problem, primarily because the KPZ equation is a \emph{singular stochastic PDE}; there are ultraviolet ``UV" divergences whose rigorous analysis is highly nontrivial and was ultimately resolved by Hairer \cite{Hai13} using a first version of what would be later developed into a theory of regularity structures \cite{Hai14}. We explain this singular nature of KPZ and related SPDEs in the context of weak universality. First, however, let us introduce the KPZ equation below, which we pose on the space-time set $\R_{\geq0}\times\mathbb{T}$, where $\mathbb{T}=\R/\Z$ is the unit one-dimensional torus, and introduce a Gaussian space-time white noise on $\mathbb{T}$, defined to be a Gaussian random field with covariance kernel formally given by $\E\xi_{\t,\x}\xi_{\s,\y}=2\delta_{\t=\s}\delta_{\x=\y}$:
\begin{align}
\partial_{\t}\mathbf{h} = \Delta\mathbf{h} + \beta|\grad\mathbf{h}|^{2} + \xi. \label{eq:KPZ}
\end{align}
The constant $\beta$ is taken to be a non-negative number; its sign is not too important as we can always take $-\mathbf{h}$ instead of $\mathbf{h}$, and the noise is statistically invariant under sign change. The $\beta$ constant will eventually be a homogenized coefficient in our main universality result. We will be mostly concerned with generalizations of \eqref{eq:KPZ} in which we replace $\beta|\grad\mathbf{h}|^{2}$ by a general function of $\grad\mathbf{h}$. In particular, let us define $\mathbf{g}^{\e}$ as the solution to the following SPDE, in which $\mathbf{F}$ is a general function and $\e$ is a small positive parameter we eventually take to zero in order to recover the KPZ equation:
\begin{align}
\partial_{\t}\mathbf{g}^{\e} = \Delta\mathbf{g}^{\e} + \e^{-1}\mathbf{F}(\e^{1/2}\grad\mathbf{g}^{\e}) + \xi. \label{eq:SPDE}
\end{align}
We eventually consider solutions to a slightly and technically modified version of \eqref{eq:SPDE}, but for transparency of illustration we consider formally \eqref{eq:SPDE}. To illustrate weak universality and singular behavior, one can Taylor expand the $\mathbf{F}$ nonlinearity on the RHS of \eqref{eq:SPDE} about its value at 0. The first two terms, which are constant and linear in $\grad\mathbf{g}^{\e}$, can be removed when we insert them back into the SPDE \eqref{eq:SPDE} by global time-shift and linear change-of-reference along a constant speed characteristic. We are left with the quadratic term in $\grad\mathbf{g}^{\e}$ plus terms that are formally vanishing in the small $\e$ limit, so we are led to believe $\mathbf{g}^{\e}$ converges to \eqref{eq:KPZ} for $\beta = 2^{-1}\mathbf{F}''(0)$. The convergence to KPZ is true, but the argument is \emph{wrong} as the effective coefficient is global in $\mathbf{F}''$. We make precise the correct effective coefficient later in this introduction, but we remark some version of both claims in the previous sentence are proved in \cite{GP16,HQ,HX}. 

In \cite{HQ,HX}, the method of proving universality is based on the theory of regularity structures in \cite{Hai14}. In \cite{GP16}, the method of proof is based on \emph{energy solution theory} \cite{GJ15, GP}. On the one hand, because regularity structures is based on pathwise and analytic considerations, the notion of convergence in \cite{HQ,HX} is much stronger than that in \cite{GP16}. Moreover, the nature of the smoothing of \eqref{eq:KPZ} is much more general than in \cite{GP16}, as \cite{GP16} strongly depends on both the Markovian structure of the smoothed SPDEs \eqref{eq:SPDE} and the explicit nature of invariant measures. On the other hand, the paper \cite{GP16} treats a much larger set of nonlinearities than \cite{HQ,HX}. For example, \cite{GP16} can treat \eqref{eq:SPDE} for $\mathbf{F}(\x)=|\x|$, but \cite{HQ,HX} are far from treating this non-smooth nonlinearity. 

Currently, however, energy solution theory \cite{GP16} is exclusively applicable to Brownian initial data. One point of this paper is to extend \cite{GP16} to arbitrary continuous initial conditions for \eqref{eq:KPZ} and extend the energy solution theory from stationary Brownian initial data to general continuous initial data; this would actually provide a first intrinsic solution theory for singular SPDEs with arbitrary continuous initial data instead of quantitatively H\"{o}lder initial data.

Before we proceed with a precise discussion of universality, however, we must introduce the following \emph{stochastic Burgers equation} (SBE), whose solution $\mathbf{u}$ is related to $\mathbf{h}$ from \eqref{eq:KPZ} by taking a weak derivative $\mathbf{u}=\grad\mathbf{h}$:
\begin{align}
\partial_{\t}\mathbf{u} = \Delta\mathbf{u} + \beta\grad(\mathbf{u}^{2}) + \grad\xi. \label{eq:SBE}
\end{align}
Technically, results of \cite{GP16} are stated at the level of \eqref{eq:SBE}, not \eqref{eq:KPZ}, but as shown in \cite{GP} the difference between \eqref{eq:KPZ} and \eqref{eq:SBE}, which is a constant eliminated via differentiation, poses no problem. In particular, our aforementioned extension of energy solutions \cite{GJ15, GP} will study \eqref{eq:SBE} instead of \eqref{eq:KPZ}, but it can treat \eqref{eq:KPZ} without much difficulty.

Having introduced the stochastic Burgers equation \eqref{eq:SBE}, let us now introduce exactly what we mean by an ``energy solution". The following is taken verbatim from the beginning of Section 2 of \cite{GP17} instead of \cite{GP}. We explicitly mention the differences between notions of energy solution in \cite{GP17} and \cite{GP} below as well.
\begin{itemize}
\item Suppose $\mathbf{u}\in\mathscr{C}(\R_{\geq0},\mathscr{D}(\mathbb{T}))$, where $\mathscr{D}(\mathbb{T})$ is the topological dual of $\mathscr{C}^{\infty}(\mathbb{T})$ (the latter is equipped with the usual Frechet topology). Suppose there exists $\rho\in\mathscr{C}^{\infty}_{c}(\mathbb{R})$ such that $\int_{\R}\rho(\x)\d\x=1$ and, for $\rho^{N}(\cdot)=N\rho(N\cdot)$, the limit
\begin{align}
\int_{0}^{\t}\grad(\mathbf{u}_{\s})^{2}\d\s(\Phi) \ := \ \lim_{N\to\infty}\int_{0}^{\t}(\mathbf{u}_{\s}\ast\rho^{N})^{2}(-\grad\Phi)\d\s
\end{align}
exists for all $\Phi\in\mathscr{C}^{\infty}(\mathbb{T})$ locally uniformly in $\t$ in probability. Above, the convolution is convolution on the torus.
\item Suppose in addition that for any $\Phi\in\mathscr{C}^{\infty}(\mathbb{T})$, the process 
\begin{align}
\mathbf{B}(\t;\Phi) \ := \ \mathbf{u}_{\t}(\Phi) - \mathbf{u}_{0}(\Phi) - \int_{0}^{\t}\mathbf{u}_{\s}(\Delta\Phi)\d\s - \int_{0}^{\t}\grad(\mathbf{u}_{\s})^{2}\d\s(\Phi)
\end{align}
is a martingale in $\t\geq0$ with quadratic variation $2\t\int_{\mathbb{T}}|\grad\Phi(\x)|^{2}\d\x$. Lastly, suppose the following energy estimate holds:
\begin{align}
\E|\int_{\s}^{\t}(\mathbf{u}_{r}\ast\rho^{N})^{2}(-\grad\Phi)-(\mathbf{u}_{r}\ast\rho^{M})^{2}(-\grad\Phi)\d r|^{2} \ \lesssim \ \frac{|\t-\s|}{M\wedge N}\int_{\mathbb{T}}|\grad\Phi(\x)|^{2}\d\x. \label{eq:energyestimatesolution}
\end{align}
\item If the two bullet points above hold, we say $\mathbf{u}$ is an energy solution to \eqref{eq:SBE}. 
\end{itemize}
In \cite{GP}, a backwards representation of the $\mathbf{u}$ process is also required. However, this backwards representation is used only for some technical estimates. As noted in \cite{GP17} (see Theorem 2.8, for example), this backwards representation can be dropped in cases that we eventually specialize to without losing this technical benefit.

In turns out that our approach to universality (Theorem \ref{theorem:KPZ1}) will be intimately connected to establishing explicit rates of convergence to the invariant measure for \eqref{eq:SBE} that were previously missing from the literature for \emph{general} continuous initial conditions; the paper \cite{GP20} studied spectral gaps that only provide information, a priori, for initial data very close to the white noise invariant measure, which turn out to be basically obtainable via appropriate discretization of \eqref{eq:SBE} by Ornstein-Uhlenbeck processes plus an asymmetric interaction. Actually, we will enhance the spectral gap in \cite{GP20} to a much more powerful log-Sobolev inequality using basically the same discretization method via Fourier smoothing. 

We conclude this introduction by noting our work successfully treats \emph{general continuous} data, going well beyond the usual class of H\"{o}lder regular initial data required to analyze SPDEs like \eqref{eq:KPZ} and \eqref{eq:SBE} via regularity structures \cite{Hai14} or paracontrolled distributions \cite{GPI}. In particular, this highlights the utility of our perspective for \eqref{eq:KPZ} and \eqref{eq:SBE}, which ultimately agree with the same solutions obtained via other methods \cite{GPI,GJara,Hai14}. It also even provides a first \emph{intrinsic} solution theory to SPDEs like \eqref{eq:KPZ} and \eqref{eq:SBE} for general continuous initial data; although a to-be-introduced Cole-Hopf solution theory also treats such initial data, this requires \cite{BG} appealing to an auxiliary SPDE.
\subsection{Universality}
We start with the following construction or ansatz that we clarify afterwards. 
\begin{definition}\label{definition:intro1}
Let us first define a ``UV projection" map $\Pi_{N}:\mathscr{L}^{2}(\mathbb{T})\to\mathscr{L}^{2}(\mathbb{T})$, in which $N$ is a positive integer, via the basis-prescription $\Pi_{N}(\exp(2\pi i k\x))=\mathbf{1}(|k|\leq N)\exp(2\pi ik\x)$ for all $k\in\Z$. We now define $\mathbf{h}^{N,1}$ as the solution to the following SPDE on $\R_{\geq0}\times\mathbb{T}$ with regularized noise:
\begin{align}
\partial_{\t} \mathbf{h}^{N,1} \ &= \ \Delta \mathbf{h}^{N,1} + \e_{N}^{-1} \Pi_{N} \mathbf{F}(\e_{N}^{1/2} \grad\mathbf{h}^{N,1}) + \Pi_{N} \xi; \label{eq:UVKPZ1}
\end{align}
In \eqref{eq:UVKPZ1}, $\e_{N}=N^{-1}\pi$ is a normalization factor used in \cite{GP16}. The function $\mathbf{F}$ is general; we specify conditions on it later in Assumption \ref{ass:intro2}.

We additionally define $\mathbf{h}^{N,2}$ to be the solution to the following noise-free PDE on $\R_{\geq0}\times\mathbb{T}$ that is ``controlled" by $\mathbf{h}^{N,1}$, in the sense that the following PDE has randomness coming only from $\mathbf{h}^{N,1}$:
\begin{align}
\partial_{\t}\mathbf{h}^{N,2} \ = \ \Delta\mathbf{h}^{N,2} + \e_{N}^{-1}\mathbf{F}(\e_{N}^{1/2}\grad(\mathbf{h}^{N,1}+\mathbf{h}^{N,2})) - \e_{N}^{-1}\mathbf{F}(\e_{N}^{1/2}\grad\mathbf{h}^{N,1}). \label{eq:UVKPZ2}
\end{align}
Lastly, we define $\mathbf{h}^{N}=\mathbf{h}^{N,1}+\mathbf{h}^{N,2}$, we define $\mathbf{u}^{N,1}=\grad\mathbf{h}^{N,1}$, and we define $\mathbf{u}^{N}=\grad\mathbf{h}^{N}=\grad\mathbf{h}^{N,1}+\grad\mathbf{h}^{N,2}$. Let us emphasize that $\mathbf{u}^{N,1}$ satisfies the regularized SPDE given below on $\R_{\geq0}\times\mathbb{T}$:
\begin{align}
\partial_{\t}\mathbf{u}^{N,1} \ = \ \Delta\mathbf{u}^{N,1} + \e_{N}^{-1}\Pi_{N}\grad\mathbf{F}(\e_{N}^{1/2}\mathbf{u}^{N,1}) + \Pi_{N}\grad\xi. \label{eq:UVSBE1}
\end{align}
\end{definition}
\begin{remark}
{Because \eqref{eq:UVKPZ2} is a nonlinear PDE, let us must mention exactly what we mean by ``solution". Precisely, we require $\mathbf{h}^{N,2}(\t,\cdot)\in\mathscr{C}^{1}(\mathbb{T})$, where $\mathscr{C}^{1}(\mathbb{T})$ is the space of continuously differentiable functions on $\mathbb{T}$. Moreover, we require $\|\mathbf{h}^{N,2}(\t,\cdot)\|_{\mathscr{C}^{1}(\mathbb{T})}$ is continuous in $\t$. We also require $\mathbf{h}^{N,2}$ to solve
\begin{align*}
\mathbf{h}^{N,2}(\t,\x) \ &= \ (\mathrm{e}^{\t\Delta}\mathbf{h}^{N,2}(0,\cdot))(\x) \\
&+ \int_{0}^{\t}\left(\mathrm{e}^{(\t-\s)\Delta}\left(\e_{N}^{-1}\mathbf{F}(\e_{N}^{1/2}\grad(\mathbf{h}^{N,1}(\s,\cdot)+\mathbf{h}^{N,2}(\s,\cdot)) - \e_{N}^{-1}\mathbf{F}(\e_{N}^{1/2}\grad\mathbf{h}^{N,1}(\s,\cdot))\right)\right)\d\s.
\end{align*}
Existence and uniqueness of such mild solutions to \eqref{eq:UVKPZ2} will be guaranteed under our assumptions on $\mathbf{F}$, as long as $\mathbf{h}^{N,2}(0,\cdot)\in\mathscr{C}^{1}(\mathbb{T})$; see Lemma \ref{lemma:pde1}.}
\end{remark}
Let us now explain Definition \ref{definition:intro1}, beginning with the $\mathbf{h}^{N,1}$ equation. Let us specialize to \eqref{eq:UVKPZ1} for $\mathbf{F}(\x)=\x^2$. In this case, in \cite{GJara, GP} it is shown that \eqref{eq:UVKPZ1} with this quadratic nonlinearity $\mathbf{F}$ is an appropriate regularization of the continuum KPZ equation \eqref{eq:KPZ} that agrees with the Cole-Hopf solution of \cite{BG} and regularity structures solution of \cite{Hai13}, \emph{if} the initial data to \eqref{eq:UVKPZ1} is a Brownian bridge on $\mathbb{T}$ {(that is independent of the space-time white noise $\xi$)} or {a} random initial data whose law as a probability measure on $\mathscr{C}(\mathbb{T})$ has relative entropy that is uniformly bounded with respect to Brownian bridge measure. Actually, as noted in \cite{FQ} the choice of $\Pi_{N}$ regularization is not so important, and we can actually use a much smoother cutoff (for example convolution by a time-1 heat kernel with Fourier cutoff at level $N$) for which the results of \cite{GJara, GP} still hold.

We proceed to explain $\mathbf{h}^{N,2}$. {Because the solution to \eqref{eq:UVSBE1} is supported on non-zero Fourier modes $|k|\leq N$, we can replace $\mathbf{u}^{N,1}$ in \eqref{eq:UVSBE1} by $\Pi_{N}^{-1}\mathbf{u}^{N,1}$. We can also replace $\Pi_{N}\grad\xi$ by $\Pi_{N}^{2}\grad\xi$. Then, applying $\Pi_{N}^{-1}$ to \eqref{eq:UVSBE1} formally gives the stochastic Burgers equation with general $\mathbf{F}$ and smoothed noise (this is formal because $\Pi_{N}$ is not invertible). Moreover, in \cite{GP16}, convergence of $\mathbf{u}^{N,1}$ means convergence of fixed Fourier modes, and applying $\Pi_{N}^{-1}$ to $\mathbf{u}^{N,1}$ does not affect this convergence. We point out that this formal relation between \eqref{eq:UVSBE1} and stochastic Burgers with general $\mathbf{F}$ can be made somewhat rigorous in the large-$N$ limit. In particular, in the quadratic case $\mathbf{F}(\x)=\x^{2}$, the $\Pi_{N}$-Fourier smoothing operator in front of $\mathbf{F}$ in the $\mathbf{h}^{N,1}$ equation \eqref{eq:UVKPZ1} does not change the large-$N$ limit of the solution $\mathbf{h}^{N,1}$, as in either case one gets the same solution to the KPZ equation; see \cite{GP,Hai13}.} Thus, let us pretend \eqref{eq:UVKPZ1} does not have $\Pi_{N}$ in front of $\mathbf{F}(\x)=\x^{2}$. Then if $\mathbf{h}^{N,2}$ solves \eqref{eq:UVKPZ2}, it is easy to see that $\mathbf{h}^{N}=\mathbf{h}^{N,1}+\mathbf{h}^{N,2}$ solves the same equation \eqref{eq:UVKPZ1} as $\mathbf{h}^{N,1}$ without $\Pi_{N}$ in front of $\mathbf{F}$, though with {different} initial data. {So for the specific regularization $\Pi_{N}$ in Definiiton \ref{definition:intro1}, we formally write the solution to a regularized stochastic Burgers equation with general nonlinearity into the two pieces \eqref{eq:UVKPZ1} and \eqref{eq:UVKPZ2}}. We finish this explanation of Definition \ref{definition:intro1} by noting the decomposition $\mathbf{h}^{N}=\mathbf{h}^{N,1}+\mathbf{h}^{N,2}$ is by no means canonical, namely if we think of $\mathbf{h}^{N}$ as a solution to a regularized SPDE. In particular, we will specify an initial data for $\mathbf{h}^{N}$, and we have a choice as to how to decompose this initial data into initial data for $\mathbf{h}^{N,1}$ and $\mathbf{h}^{N,2}$ equations. We clarify this in Theorem \ref{theorem:KPZ1}.

Let us now introduce the necessary assumptions for the nonlinearity $\mathbf{F}$ in Definition \ref{definition:intro1} and the initial data of consideration for $\mathbf{h}^{N}$, which we re-emphasize is, technically, a pair of initial data for $\mathbf{h}^{N,1}$ and $\mathbf{h}^{N,2}$, respectively. We will provide explanation afterwards.
\begin{ass}\label{ass:intro2}
The nonlinearity $\mathbf{F}:\R\to\R$ in \emph{Definition \ref{definition:intro1}} is {uniformly Lipschitz}. Moreover, if $\mathbf{P}^{\mathrm{G}}$ is standard Gaussian measure $\mathscr{N}(0,1)$ on $\R$, then $\mathbf{F},\mathbf{F}'\in\mathscr{L}^{2}(\R,\mathbf{P}^{\mathrm{G}})$.
\end{ass}
\begin{ass}\label{ass:intro3}
Assume $\mathbf{h}^{N}(0,\cdot)=\Pi_{N}\mathbf{h}(0,\cdot)$, where $\mathbf{h}(0,\cdot)$ is a continuous function on $\mathbb{T}$ {independent of $\xi$} such that $\mathbf{h}(0,0)=0$, where $0\in\mathbb{T}$ is interpreted by embedding $\mathbb{T}=\R/\Z\simeq[0,1)$, and $\mathbf{h}^{N}(0,\cdot)=\Pi_{N}\mathbf{h}(0,\cdot)\to\mathbf{h}(0,\cdot)$ uniformly on $\mathbb{T}$ with probability 1. 
\end{ass}
Let us first explain Assumption \ref{ass:intro2}. It is not difficult to see \cite{GJara} that the Fourier coordinates of $\mathbf{u}^{N,1}$ in \eqref{eq:UVSBE1} solves a finite-dimensional SDE. The Lipschitz assumption guarantees this SDE has global solutions in time. In particular, we could assume a.e. differentiability of $\mathbf{F}$ instead and prove results until blow-up times, but we choose to take the Lipschitz assumption for convenience and clarity of presentation. {The \emph{uniformly} Lipschitz assumption is also used for deriving global-in-time solutions to \eqref{eq:UVKPZ2} for each fixed $N$. Similarly, we could impose locally Lipschitz instead of uniformly Lipschitz and work until a random blow-up time (for classical solutions) for \eqref{eq:UVKPZ2}. We choose uniformly Lipschitz again for convenience and clarity. There are other assumptions besides uniformly Lipschitz, however, that would ensure (classical) global solutions for both \eqref{eq:UVSBE1} and \eqref{eq:UVKPZ2}. For example, establishing global solutions for \eqref{eq:UVSBE1} requires only locally Lipschitz $\mathbf{F}$, since the invariant measure for \eqref{eq:UVSBE1} extends local solutions to global ones; see Lemma \ref{lemma:pde2}. Global solutions for \eqref{eq:UVKPZ2}, which is a viscous Hamilton-Jacobi equation with random Hamiltonian, can be obtained for continuously differentiable $\mathbf{F}$ satisfying a polynomial bound for any arbitrary but fixed degree; see Theorem 3.1 in \cite{BA}. (We note \cite{BA} is not published, and the proof of Theorem 3.1 therein is quite complicated. Also, in \cite{HX}, the specific choice $\mathbf{F}(\x)=|\x|$ was mentioned as an interesting problem, and this choice satisfies Assumption \ref{ass:intro2}. For these reasons, we give details just for uniformly Lipschitz $\mathbf{F}$.) It also seems possible to remove the polynomial bound on $\mathbf{F}$ and upgrade its regularity to $\mathscr{C}^{1,\alpha}$. In any case, all these choices for conditions on $\mathbf{F}$ would generalize what is done in \cite{HX}.}

Moreover, the condition $\mathbf{F},\mathbf{F}'\in\mathscr{L}^{2}(\R,\mathbf{P}^{\mathrm{G}})$ comes from the assumption on nonlinearities in \cite{GP16}, which serves as a stationary model of what is considered herein. Technically, it means that the Hermite polynomial expansions, or equivalently eigenfunction expansion in $\mathscr{L}^{2}(\R,\mathbf{P}^{\mathrm{G}})$ with respect to the Ornstein-Uhlenbeck operator of Malliavin calculus, converges in $\ell^{2}(\Z)$ with both the standard constant weight and with the weight $|k|$. This is similar to how a function and its derivative belong to $\mathscr{L}^{2}(\mathbb{T})$ given its Fourier series converges in $\ell^{2}(\Z)$ after we multiply its $k$-th Fourier coordinate by $1+|k|$. {Let us make this precise. By Gaussian integration-by-parts, we get $2^{-1}\E^{\mathrm{G}}\mathbf{F}''=2^{-1}\E^{\mathrm{G}}\mathbf{F}\mathsf{H}_{2}$, where $\mathsf{H}_{2}$ is the second Hermite polynomial with respect to the Gaussian measure $\mathbf{P}^{\mathrm{G}}$ as in \cite{GP16}. In particular, we note $\E^{\mathrm{G}}\mathbf{F}''=\E^{\mathrm{G}}\mathbf{F}\mathsf{H}_{2}<\infty$. Indeed, the Hermite polynomial $\mathsf{H}_{2}$ is a unit-length Hermite basis vector in $\mathscr{L}^{2}(\R,\mathbf{P}^{\mathrm{G}})$, so $\E^{\mathrm{G}}\mathbf{F}\mathsf{H}_{2}\leq(\E^{\mathrm{G}}\mathbf{F}^{2})^{1/2}$, and this last second moment is finite by Assumption \ref{ass:intro2}.}

Let us now explain Assumption \ref{ass:intro3}. We first note that it only requires almost sure continuity of the initial data function $\mathbf{h}(0,\cdot)$, instead of quantitative H\"{o}lder regularity as in \cite{GP16,HQ,HX}, making Assumption \ref{ass:intro3} quite general. Observe that $\mathbf{h}(0,\cdot)$ is allowed to be random; moreover, if $\mathbf{h}(0,\cdot)$ is a Brownian bridge on $\mathbb{T}$, defined by independent Gaussian random variables of variance $1+|k|^{-2}$ for $k\neq0$ Fourier coordinate {and assumed to be independent of $\xi$}, then the convergence $\Pi_{N}\mathbf{h}(0,\cdot)\to\mathbf{h}(0,\cdot)$ is classical. We clarify the $k=0$ Fourier coefficient for Brownian bridge {just has to be chosen independent of the $k\neq0$ coefficients and of $\xi$}. We emphasize, again, that we can replace $\Pi_{N}$ with a smoother Fourier cutoff if desired; this would make Assumption \ref{ass:intro3} and the required uniform convergence less strict; let us remark that some convergence is needed in Assumption \ref{ass:intro3} \cite{HQ,HX}, as otherwise large-$N$ limits of $\mathbf{h}^{N}$ are out of the question. Lastly, when we refer to Assumption \ref{ass:intro3} in Theorem \ref{theorem:KPZ1}, we will specify ways of decomposing initial data $\mathbf{h}^{N}$ into initial data for $\mathbf{h}^{N,1}$ and $\mathbf{h}^{N,2}$ equations of Definition \ref{definition:intro1}.

We now present the first main theorem of this article. To this end, we must first specify a theory of solutions to \eqref{eq:KPZ} and \eqref{eq:SBE} without regularization; note Definition \ref{definition:intro1} addresses only Fourier regularized (S)PDEs. We will employ the following \emph{Cole-Hopf solution theory}, which agrees with regularity structures \cite{Hai14}, paracontrolled distributions \cite{GPI}, and energy solution theory \cite{GJara,GP}, at least for sufficiently H\"{o}lder continuous initial data (and for energy solutions, only for initial data with uniformly bounded ``relative entropy" with respect to Brownian bridge initial data; see the definition of relative entropy in Section \ref{section:re}, though we will re-clarify this when more relevant). The Cole-Hopf solution to \eqref{eq:KPZ} is defined (and explained) as follows.
\begin{itemize}
\item { The Cole-Hopf solution is defined as $\mathbf{h}=\beta^{-1}\log\mathbf{z}$, where $\partial_{\t}\mathbf{z}=\Delta\mathbf{z}+\beta\mathbf{z}\xi$ is a linear SPDE on $\R_{\geq0}\times\mathbb{T}$ called the \emph{stochastic heat equation} (SHE).}
\item { In \cite{Mu}, an almost sure comparison principle was shown for SHE. Thus, the logarithm is well-defined as long as we start $\mathbf{z}$ with positive (continuous) initial data.} 
\item {We define the Cole-Hopf solution to the stochastic Burgers equation \eqref{eq:SBE} to be $\mathbf{u}=\grad\mathbf{h}=\beta^{-1}\log\mathbf{z}$, where the derivative is interpreted in the weak sense.}
\end{itemize}
{We clarify that the Cole-Hopf solution agrees with the solution of KPZ via regularity structures only if we include a divergent counter-term in \eqref{eq:KPZ}. In particular, if we were to claim the convergence of $\mathbf{h}^{N,1}$ in Definition \ref{definition:intro1} to the Cole-Hopf solution of \eqref{eq:KPZ}, we would need to introduce a term of the form $-C_{N}$ to \eqref{eq:UVKPZ1}, where $C_{N}$ is constant and $|C_{N}|\to\infty$ as $N\to\infty$. However, our results are for the stochastic Burgers equation, for which this constant $C_{N}$ plays no role because we take its spatial derivative. Lastly, we note that the space $\mathscr{C}(\R_{\geq0},\mathscr{D}(\mathbb{T}))$ of continuous $\mathscr{D}(\mathbb{T})$-valued paths is equipped with the locally uniform (in $\R$) topology. Here, $\mathscr{D}(\mathbb{T})$ is the space of generalized functions on $\mathbb{T}$, or equivalently the topological dual of $\mathscr{C}^{\infty}(\mathbb{T})$}.
\begin{theorem}\label{theorem:KPZ1}
Consider initial data $\mathbf{h}(0,\cdot)$ as in \emph{Assumption \ref{ass:intro3}}, and { assume $\E^{\mathrm{G}}\mathbf{F}''\neq0$}. There exists a one-parameter family of decompositions $\mathbf{h}^{N}=\mathbf{h}^{N,1,\e}+\mathbf{h}^{N,2,\e}$, as in \emph{Definition \ref{definition:intro1}}, parameterized by $\e>0$ such that:
\begin{itemize}
\item For any $\e>0$, the function $\mathbf{h}^{N,1,\e}$ solves \emph{\eqref{eq:UVKPZ1}} with initial data given by $\Pi_{N}$ acting on Brownian bridge on $\mathbb{T}$ {independent of $\xi$ and} conditioned to be within $\e$ of $\mathbf{h}(0,\cdot)$. Moreover, the gradient $\mathbf{u}^{N,1,\e}=\grad\mathbf{h}^{N,1,\e}$ converges as $N\to\infty$ to the Cole-Hopf solution of \emph{\eqref{eq:SBE}} with $\beta=2^{-1}\E^{\mathrm{G}}\mathbf{F}''$ and with initial data given by the weak derivative of Brownian bridge {independent of $\xi$ and} conditioned to be within $\e$ of $\mathbf{h}(0,\cdot)$. Here, $\E^{\mathrm{G}}$ is expectation with respect to Gaussian measure $\mathbf{P}^{\mathrm{G}}$ in \emph{Assumption \ref{ass:intro2}}.
\item For any $\e>0$, $\mathbf{h}^{N,2,\e}$ solves \emph{\eqref{eq:UVKPZ2}} with initial data given by the difference $\mathbf{h}^{N}(0,\cdot)-\mathbf{h}^{N,1,\e}(0,\cdot)$. Moreover, we have $|\mathbf{h}^{N,2,\e}(\t,\x)|\leq\e+\mathrm{o}_{N}$ for all $\t\geq0$ and $\x\in\mathbb{T}$ with probability 1, where $\mathrm{o}_{N}\to0$ almost surely and uniformly in $\t\geq0$ and $\x\in\mathbb{T}$. 
\end{itemize}
{We deduce that the law of $\mathbf{u}^{N}=\grad\mathbf{h}^{N}=\grad\mathbf{h}^{N,1,\e}+\grad\mathbf{h}^{N,2,\e}$, as a probability measure on $\mathscr{C}(\R_{\geq0},\mathscr{D}(\mathbb{T}))$, converges to the Cole-Hopf solution of \emph{\eqref{eq:SBE}} with initial data given by the weak derivative $\grad\mathbf{h}(0,\cdot)$.}
\end{theorem}
Theorem \ref{theorem:KPZ1} claims that we can approximate the regularization $\mathbf{h}^{N}$ of the KPZ equation \emph{uniformly} in $N$ by a solution $\mathbf{h}^{N,1,\e}$ to \eqref{eq:UVKPZ1}, which we will ultimately analyze via {the work in \cite{GP16}}. The residual term $\mathbf{h}^{N,2,\e}$ solves \eqref{eq:UVKPZ2}; we emphasize the nontrivial aspect of this statement is two-fold. First, we must find suitable approximations to $\mathbf{h}^{N}$ that can be analyzed by using energy solution theory \cite{GJara, GP, GP16, GP17}. Since energy solution theory is mostly dependent on initial data and invariant measures, this approximation will be made at the level of initial data. The problem we are left with is to \emph{globally propagate} this approximation uniformly in $N$. In particular, we must control $\mathbf{h}^{N,2,\e}$, which solves a type of Hamilton-Jacobi equation with divergent Hamiltonian. Note that we must also extend \cite{GP16} to initial conditions that are not exactly Brownian bridge but relative entropy perturbations as well, though this will not be difficult provided the work of \cite{GJS15, GP17}. We clarify Theorem \ref{theorem:KPZ1} gives a first application of the energy solution theory to a very general class of non-stationary SPDE growth models; in particular, we only require the modest requirements in \cite{GP} for the models of interest that are much more general than the required conditions on the nonlinearity in \cite{HQ, HX}. In particular, one can \emph{define} the large-$N$ limit of $\mathbf{u}^{N}=\grad\mathbf{h}^{N}$ in Theorem \ref{theorem:KPZ1} as an \emph{intrinsic} solution to \eqref{eq:SBE} for such general continuous initial data.

{Note we have assumed $\E^{\mathrm{G}}\mathbf{F}''\neq0$. It is certainly possible that this does not hold, i.e. $\E^{\mathrm{G}}\mathbf{F}''=0$. In this case, the limit SPDE is not the Cole-Hopf solution of \eqref{eq:SBE} but rather the solution to the linear SPDE given by \eqref{eq:SBE} with $\beta=0$. The proof of Theorem \ref{theorem:KPZ1} given in this paper covers this case as well, since our method is by comparison to very-close-to-stationary initial data, and the result for stationary initial data in \cite{GP16} includes the case $\beta=0$.}
\subsection{Convergence to Invariant Measure}
As a byproduct of our proof of Theorem \ref{theorem:KPZ1}, we will be able to show that {the stochastic Burgers equation (SBE)}, starting with the gradient of any continuous initial data, converges to the Gaussian white noise measure on $\mathscr{D}(\mathbb{T})$ in a quantitative fashion. Moreover, we provide an explicit rate of convergence; this is new as \cite{HM} gives only uniqueness and ergodicity of white noise invariant measure for SBE, and \cite{GP20} provides a spectral gap estimate, which only provides quantitative bounds for initial data to SBE that is very close to the white noise invariant measure. 

{We now state the main result of this subsection. In Theorem \ref{theorem:KPZ2}, the metric $\mathsf{W}$ denotes Wasserstein distance in Appendix \ref{section:w}, for which we take the Polish space $\mathscr{X}$ to be the topological dual of the Sobolev space on $\mathbb{T}$ of bounded functions with bounded derivative. The object $\mathsf{H}$ is the relative entropy of Appendix \ref{section:re}, in which the state space is $\mathscr{S}=\mathscr{D}(\mathbb{T})$. We clarify the statement of the following result afterwards.}
\begin{theorem}\label{theorem:KPZ2}
Take any deterministic function $\mathbf{h}(0,\cdot)\in\mathscr{C}(\mathbb{T})$ and let $\mathbf{u}$ be the Cole-Hopf solution to \emph{\eqref{eq:SBE}} with initial data $\mathbf{u}(0,\cdot)=\grad\mathbf{h}(0,\cdot)$. If $\mathbf{L}(\t)$ denotes the law of $\mathbf{u}(\t,\cdot)$ as a probability measure on $\mathscr{D}(\mathbb{T})$ and if $\eta$ denotes Gaussian white noise invariant measure, then for any positive $\e$, there exists a measure $\mathbf{L}^{\e}(\t)$ such that
\begin{align}
\mathsf{W}(\mathbf{L}(\t),\mathbf{L}^{\e}(\t)) \ \leq \ \e \quad \mathrm{and} \quad \mathsf{H}(\mathbf{L}^{\e}(\t)|\eta) \ \leq \ \exp(-C\t)\kappa. \label{eq:KPZ2}
\end{align}
The constant $C$ is independent of all other parameters in this statement. The constant $\kappa=\kappa(\e,\mathbf{h})$ depends only on the modulus of continuity of $\mathbf{h}(0,\cdot)$ and $\e$. For example, if $\mathbf{h}(0,\cdot)$ is H\"{o}lder continuous, then $\kappa(\e,\mathbf{h})$ would depend only on the H\"{o}lder norm of $\mathbf{h}(0,\cdot)$ and on $\e$.
\end{theorem}
We emphasize the upper bounds in \eqref{eq:KPZ2} both vanish if we take $\t\to\infty$ then $\e\to0$.

Theorem \ref{theorem:KPZ2} builds on the spectral gap of \cite{GP20}, first by improving to a log-Sobolev inequality (LSI) via relatively straightforward considerations. Indeed, if $\grad\mathbf{h}(0,\cdot)$ has a measure with an $\mathscr{L}^{2}$-Radon-Nikodym derivative with respect to white noise measure, then $\kappa$ {equals the} $\mathscr{L}^{2}$-norm of said Radon-Nikodym derivative and $\e=0$ in \eqref{eq:KPZ2} would be okay. The point is that having an $\mathscr{L}^{2}$-Radon-Nikodym derivative is an extremely strong condition that, although is satisfied for a ``dense" class of initial data, is not satisfied for ``generic" initial data to SBE. The Wasserstein bound in \eqref{eq:KPZ2} says to first approximate gradients of general continuous initial data by a member of the aforementioned dense class. As with Theorem \ref{theorem:KPZ1}, we then globally propagate this comparison. This will give the Wasserstein estimate in \eqref{eq:KPZ2}. Ultimately, Theorem \ref{theorem:KPZ2} says that the law of the Cole-Hopf solution to \eqref{eq:SBE} with initial data given by the (weak) derivative of any general continuous function is within $\e$ of a solution that is statistically close to white noise after long times via LSI, but the comparison with general data is in a different (Wasserstein) topology.
\subsection{Organization}
We first prove Theorem \ref{theorem:KPZ1}; the proof of Theorem \ref{theorem:KPZ2} will follow from steps in the proof of Theorem \ref{theorem:KPZ1} except for a log-Sobolev inequality for the stochastic Burgers equation and its $\Pi_{N}$-UV cutoffs. Lastly, before the appendix sections, which give definitions and basic properties of key functionals we use for the reader's convenience, we comment on extensions of our methods to fractional stochastic Burgers equations.
\subsection{Acknowledgements}
The author would like to thank Amir Dembo for helpful discussion and advice. The author thanks the Northern California Chapter of the ARCS Foundation, under whose funding this research was conducted.
%
%
%
\section{Proof of Theorem \ref{theorem:KPZ1}}
The main objectives of this section, which are the main ingredients to proving Theorem \ref{theorem:KPZ1}, are listed below.
\begin{itemize}
\item We begin with a construction of $\mathbf{h}^{N,1,\e}$, in which $\e$ is a fixed positive parameter. In particular, we specify the initial condition for $\mathbf{h}^{N,1,\e}$, which then determines it uniquely because it must then solve \eqref{eq:UVKPZ1}. Observe that $\mathbf{h}^{N,2,\e}$ is automatically determined once we construct $\mathbf{h}^{N,1,\e}$ and its initial condition, because {the} initial condition for $\mathbf{h}^{N,2,\e}$ is then given by the difference between the initial data $\mathbf{h}^{N}(0,\cdot)=\Pi_{N}\mathbf{h}(0,\cdot)$ and the initial data $\mathbf{h}^{N,1,\e}(0,\cdot)$. Again, once we specify initial data of $\mathbf{h}^{N,2,\e}$, we specify $\mathbf{h}^{N,2,\e}$ uniquely because it must solve the equation \eqref{eq:UVKPZ2}, which has a unique solution because it is a classical parabolic Hamilton-Jacobi-type equation \cite{Evans}.
\item We then proceed to establish the two bullet points in Theorem \ref{theorem:KPZ1}; the last statement therein will follow by a standard argument. For $\mathbf{h}^{N,2,\e}$, we employ a maximum principle whereas for $\mathbf{h}^{N,1,\e}$, we employ the energy solution theory in \cite{GP16} combined with relative entropy estimates as with \cite{GJS15, GP17}; for this, we must prove a relative entropy estimate for the law of $\mathbf{h}^{N,1,\e}(0,\cdot)$ with respect to Brownian bridge initial data.
\item First assume Theorem \ref{theorem:KPZ1} were true for \emph{deterministic} continuous initial data $\mathbf{h}(0,\cdot)$ for which $\Pi_{N}\mathbf{h}(0,\cdot)$ converges to $\mathbf{h}(0,\cdot)$ uniformly on $\mathbb{T}$. We would be able to show Theorem \ref{theorem:KPZ1} for the allowed random continuous initial data in Assumption \ref{ass:intro3} by conditioning on said random initial data, tossing out probability zero events. Indeed, convergence as measures is a deterministic statement, namely a statement about convergence of expectations of test functions. Convergence almost surely is also deterministic if we toss out all probability zero events. We will therefore assume that $\mathbf{h}(0,\cdot)$ is deterministic and continuous, and that $\Pi_{N}\mathbf{h}(0,\cdot)$ converges to $\mathbf{h}(0,\cdot)$ uniformly on $\mathbb{T}$.
\end{itemize}
\subsection{Construction of $\mathbf{h}^{N,1,\e}$}
Let us recall the initial data $\mathbf{h}^{N}(0,\cdot)=\Pi_{N}\mathbf{h}(0,\cdot)$ converges uniformly to $\mathbf{h}(0,\cdot)$ with probability 1. Thus, to approximate $\mathbf{h}^{N}(0,\cdot)$ initial data within $\e$, which we emphasize is deterministic, it suffices to approximate the limit $\mathbf{h}(0,\cdot)$ itself. We will do so in the sequel for which we provide explanation afterwards.
\begin{definition}\label{definition:KPZ11}
Given any deterministic data $\mathbf{h}(0,\cdot)$, define the Fourier smoothing $\mathbf{h}^{N,1,\e}(0,\cdot)=\Pi_{N}\mathbf{h}^{1,\e}(0,\cdot)$, where $\mathbf{h}^{1,\e}(0,\cdot)$ is the Brownian bridge $\mathbf{b}$ {independent of $\xi$} conditioned to be within $\e$ of $\mathbf{h}(0,\cdot)$ uniformly on $\mathbb{T}$; observe that this construction implicitly chooses the $k=0$ Fourier coefficient for Brownian bridge so that this event on which we condition has positive probability. Let $\mathbf{P}^{\e}$ be the law of $\mathbf{h}^{1,\e}$ as a measure on $\mathscr{C}(\mathbb{T})$, and let $\mathbf{P}^{\infty}$ be Brownian bridge measure on $\mathscr{C}(\mathbb{T})$.
\end{definition}
Definition \ref{definition:KPZ11} can be viewed as constructing $\mathbf{h}^{N,1,\e}$ to automatically satisfy the first required bullet point in Theorem \ref{theorem:KPZ1} at the level of initial data. We will show in this section that this comparison from Theorem \ref{theorem:KPZ1} propagates globally in time. However, to clarify why we choose Brownian approximations, we appeal to the following result, which shows that the associated probability measure $\mathbf{P}^{\e}$ is uniformly ``stable" with respect to the law of Brownian bridge itself free of any conditioning. This stability is at the level of relative entropy; see Definition \ref{definition:re1} for its definition and Lemma \ref{lemma:re2} for the key properties of relative entropy that we will use to study $\mathbf{h}^{N,1,\e}$.
\begin{lemma}\label{lemma:KPZ12}
There exists a constant $\kappa=\kappa(\e,\mathbf{h}(0,\cdot))$ depending only on $\e$ and $\mathbf{h}(0,\cdot)$, in particular independent of $N$, such that the relative entropy of $\mathbf{P}^{\e}$ with respect to $\mathbf{P}^{\infty}$ is bounded by $\kappa$:
\begin{align}
\mathsf{H}(\mathbf{P}^{\e}|\mathbf{P}^{\infty}) \ \leq \ \kappa.
\end{align}
The same is true is we replace $(\mathbf{P}^{\e},\mathbf{P}^{\infty})$ by $(\mathscr{T}_{\ast}\mathbf{P}^{\e},\mathscr{T}_{\ast}\mathbf{P}^{\infty})$, where $\mathscr{T}_{\ast}$ is the pushforward on probability measures {induced by differentiation $\grad$, Fourier smoothing $\Pi_{N}$, or compositions of these two}. { The same is also true if we replace $(\mathbf{P}^{\e},\mathbf{P}^{\infty})$ with path-space measures induced by \emph{\eqref{eq:UVSBE1}}, when viewed as a finite-dimensional SDE parameterized by the image of $\Pi_{N}$, with initial measures $(\mathscr{T}_{\ast}(\Pi_{N})_{\ast}\mathbf{P}^{\e},\mathscr{T}_{\ast}(\Pi_{N})_{\ast}\mathbf{P}^{\infty})$ for the same choices of $\mathscr{T}_{\ast}$.}
\end{lemma}
\begin{proof}
It suffices to prove the relative entropy bound for $\mathbf{P}^{\e}$ with respect to $\mathbf{P}^{\infty}$, as the other relative entropy estimates follow for free by Lemma \ref{lemma:re2}. To this end, observe the function $\x\log\x$ is continuous on $\R_{\geq0}$. Because relative entropy in Definition \ref{definition:re1} is controlled by the maximum of $\x\log\x$ for $\x$ equal to the Radon-Nikodym derivative of $\mathbf{P}^{\e}$ with respect to $\mathbf{P}^{\infty}$, it then suffices to estimate this Radon-Nikodym derivative uniformly in the randomness in the expectation defining relative entropy. Moreover, since $\mathbf{P}^{\e}$ is defined as conditioning $\mathbf{P}^{\infty}$ on an event that we denote by $\mathcal{E}(\e)$, it suffices to estimate $\mathcal{E}(\e)$-probability under $\mathbf{P}^{\infty}$ from below by a $\e,\mathbf{h}(0,\cdot)$-dependent constant. To make this precise, we note:
\begin{align}
\mathsf{H}(\mathbf{P}^{\e}|\mathbf{P}^{\infty}) \ \leq \ \|\mathbf{R}^{\e}\log\mathbf{R}^{\e}\|_{\omega;\infty} \ \leq \ \sup_{0\leq\x\leq\mathbf{P}^{\infty}(\mathcal{E}(\e))^{-1}} |\x\log\x| \ \lesssim_{\mathbf{P}^{\infty}(\mathcal{E}(\e))^{-1}} \ 1. \label{eq:KPZ121}
\end{align}
{ In \eqref{eq:KPZ121}, $\mathbf{R}^{\e}$ is the Radon-Nikodym derivative of $\mathbf{P}^{\e}$ with respect to $\mathbf{P}^{\infty}$ where $\|\|_{\omega;\infty}$ is the $\infty$-norm for functions on $\mathscr{D}(\mathbb{T})$. The last bound in \eqref{eq:KPZ121} follows from the fact that $\x\log\x$ restricted to $[0,C]$ is uniformly controlled by a $C$-dependent constant; to clarify, we note $\mathbf{R}^{\e}\leq\mathbf{P}^{\infty}(\mathcal{E}(\e))^{-1}$ by construction via conditioning in Definition \ref{definition:KPZ11}.}

In particular, to control the far LHS of \eqref{eq:KPZ121} in terms of $\e,\mathbf{h}(0,\cdot)$, it suffices to bound $\mathbf{P}^{\infty}(\mathcal{E}(\e))^{-1}$ from above in terms of $\e,\mathbf{h}(0,\cdot)$, and therefore $\mathbf{P}^{\infty}(\mathcal{E}(\e))$ from below in terms of $\e,\mathbf{h}(0,\cdot)$. To this end, take a deterministic smooth function $\psi_{\e}$ that is uniformly within $\e/2$ of $\mathbf{h}(0,\cdot)$; because $\mathbf{h}(0,\cdot)$ is uniformly continuous, we may take $\psi_{\e}$ whose derivatives are controlled by $\e,\mathbf{h}(0,\cdot)$-dependent constants. As $\mathcal{E}(\e)$ is the event on which a Brownian bridge $\mathbf{b}$ is uniformly within $\e$ of $\mathbf{h}(0,\cdot)$, we have containment $\mathcal{E}(\e)\supseteq\mathcal{E}'(\e)$ where $\mathcal{E}'(\e)$ is the event on which $\psi_{\e}$ and $\mathbf{b}$ are uniformly within $\e/2$ of each other, or equivalently $\mathbf{b}-\psi_{\e}$ is uniformly bounded by $\e/2$ in absolute value. As $\psi_{\e}$ is smooth with $\e,\mathbf{h}(0,\cdot)$-dependent derivatives, the Girsanov theorem tells us that the Radon-Nikodym derivative of $\mathbf{b}-\psi_{\e}$ with respect to Brownian bridge measure is $\mathscr{L}^{2}$ as a function of $\mathscr{C}(\mathbb{T})$; here, we think of both $\mathbf{b}-\psi_{\e}$ and $\mathbf{b}$ as SDEs parameterized by $\x\in\mathbb{T}$. Thus, if $\wt{\mathbf{R}}^{\e}$ denotes this last Radon-Nikodym derivative, the Cauchy-Schwarz inequality tells us the following probability estimate in which $\|\|_{\omega;2}$ denotes the $\mathscr{L}^{2}$-norm for functions on $\mathscr{D}(\mathbb{T})$ with respect to the law of $\mathbf{b}-\psi_{\e}$:
\begin{align}
{\mathbf{P}^{\infty}}\left({\sup}_{x\in\mathbb{T}}|\mathbf{b}(x)|\leq\e/2\right) \ \leq \ \|\wt{\mathbf{R}}^{\e}\|_{\omega;2}\left(\mathbf{P}^{\infty}\left(\mathcal{E}'(\e)\right)\right)^{1/2}. \label{eq:KPZ122}
\end{align}
The reflection principle for Brownian bridge implies the LHS of \eqref{eq:KPZ122} is controlled by an $\e$-dependent constant from below. An upper bound on $\|\wt{\mathbf{R}}^{\e}\|_{\omega;2}$ therefore bounds $\mathbf{P}^{\infty}(\mathcal{E}'(\e))$ from below by $\e,\mathbf{h}(0,\cdot)$-dependent constants, which completes the proof via \eqref{eq:KPZ121} and the sentence immediately after \eqref{eq:KPZ121}.
\end{proof}
\subsection{Convergence of $\mathbf{u}^{N,1,\e}$}
In this subsection, we prove that $\mathbf{u}^{N,1,\e}=\grad\mathbf{h}^{N,1,\e}$ converges to the solution of \eqref{eq:SBE} with the initial data claimed in Theorem \ref{theorem:KPZ1}. Roughly speaking, this is done in \cite{GP16} if the initial data of $\mathbf{h}^{N,1,\e}$ is a Brownian bridge {independent of $\xi$ and} free of any conditioning. This Brownian bridge initial data is ultimately used, however, to show that certain terms vanish and that certain limits exist in the large-$N$ limit. Therefore, the relative entropy estimate in Lemma \ref{lemma:KPZ12} allows us to inherit all these for $\mathbf{h}^{N,1,\e}$ and $\mathbf{u}^{N,1,\e}$ initial data in Definition \ref{definition:KPZ11}. We formally state this in the following.
\begin{lemma}\label{lemma:KPZ14}
Given any fixed positive $\e$, the process $\mathbf{u}^{N,1,\e}=\grad\mathbf{h}^{N,1,\e}$ converges {weakly in $\mathscr{C}(\R_{\geq0},\mathscr{D}(\mathbb{T}))$} to the {Cole-Hopf} solution of \emph{\eqref{eq:SBE}} with initial data given by the weak derivative limit $\lim_{N\to\infty}\grad\mathbf{h}^{N,1,\e}(0,\cdot)=\grad\mathbf{h}^{1,\e}(0,\cdot)$, which is the weak derivative of Brownian bridge {independent of $\xi$ and} conditioned to stay uniformly within $\e$ on $\mathbb{T}$ of $\mathbf{h}(0,\cdot)$ constructed in \emph{Definition \ref{definition:KPZ11}}.
\end{lemma}
\begin{proof}
{Note that if the process $\mathbf{u}^{N,1,\e}=\grad\mathbf{h}^{N,1,\e}$ converges weakly in $\mathscr{C}(\R_{\geq0},\mathscr{D}(\mathbb{T}))$, its initial data must be given by the limit $\lim_{N\to\infty}\grad\mathbf{h}^{N,1,\e}(0,\cdot)=\grad\mathbf{h}^{1,\e}(0,\cdot)$. We clarify this limit is taken with respect to weak-$\ast$ topology on $\mathscr{D}(\mathbb{T})$. It can be computed by first recalling $\lim_{N\to\infty}\mathbf{h}^{N,1,\e}(0,\cdot)=\lim_{N\to\infty}\grad\Pi_{N}\mathbf{h}^{1,\e}(0,\cdot)$. Because $\grad$ is continuous with respect to weak-$\ast$ topology on $\mathscr{D}(\mathbb{T})$, we know $\lim_{N\to\infty}\grad\Pi_{N}\mathbf{h}^{N,1,\e}(0,\cdot)=\grad\lim_{N\to\infty}\Pi_{N}\mathbf{h}^{1,\e}(0,\cdot)$. Because the law of $\mathbf{h}^{1,\e}$ is absolutely continuous with respect to Brownian bridge $\mathbf{b}$ on $\mathbb{T}$, and because $\Pi_{N}\mathbf{b}\to\mathbf{b}$ uniformly on $\mathbb{T}$ with probability 1, we deduce $\Pi_{N}\mathbf{h}^{1,\e}\to\mathbf{h}^{1,\e}$ uniformly on $\mathbb{T}$ with probability 1 as well. Since uniform topology is stronger than weak-$\ast$ topology, we deduce $\lim_{N\to\infty}\Pi_{N}\mathbf{h}^{1,\e}(0,\cdot)=\mathbf{h}^{1,\e}(0,\cdot)$. Ultimately, we obtain that the initial data of $\mathbf{u}^{\infty,1,\e}$, if it converges weakly in $\mathscr{C}(\R_{\geq0},\mathscr{D}(\mathbb{T}))$, is equal to (in law) $\mathbf{h}^{1,\e}$.}

It suffices to prove convergence of $\mathbf{u}^{N,1,\e}=\grad\mathbf{h}^{N,1,\e}$ to the Cole-Hopf solution of \eqref{eq:SBE}. For this, we establish notation, setting $\langle,\rangle_{\mathbb{T}}$ as the integral pairing for functions on $\mathbb{T}$ and $\langle,\rangle_{\t;\mathbb{T}}=\int_{0}^{\t}\langle,\rangle\d\s$. To prove convergence, we first will show tightness of $\mathbf{u}^{N,1,\e}$ and then identify limit points. {To this end, we have the proposed tightness for Brownian bridge initial data in \cite{GP16}. Lemma \ref{lemma:KPZ12} gives $\mathrm{O}(1)$-relative entropy for the path-space law of $\mathbf{u}^{N,1,\e}$ with respect to the path-space law with Brownian bridge initial data. The last paragraph Lemma \ref{lemma:re2} then lets us inherit the proposed tightness from that for Brownian bridge initial data. Ultimately, we deduce tightness of $\mathbf{u}^{N,1,\e}$ in $\mathscr{C}(\R_{\geq0},\mathscr{D}(\mathbb{T}))$.}

To identify limit points, the first step here is to use the It$\hat{\mathrm{o}}$ formula to obtain the following in which $\mathbf{B}^{N}(\t;\Phi)$ is a Brownian motion of quadratic variation $2\t\|\grad\Pi_{N}\Phi\|_{\x;2}^{2}$, where $\|\|_{\x;2}$ is the $\mathscr{L}^{2}(\mathbb{T})$-norm:
\begin{align}
\langle\mathbf{u}^{N,1,\e}(\t,\cdot),\Phi\rangle_{\mathbb{T}} \ = \ \langle\mathbf{u}^{N,1,\e}(0,\cdot),\Phi\rangle_{\mathbb{T}} + \langle\mathbf{u}^{N,1,\e},\Delta\Phi\rangle_{\t;\mathbb{T}} - \e_{N}^{-1}\langle\Pi_{N}\mathbf{F}(\e_{N}^{1/2}\mathbf{u}^{N,1,\e}), \grad\Phi\rangle_{\t;\mathbb{T}} + \mathbf{B}^{N}(\t;\Phi). \label{eq:KPZ141}
\end{align}
{Observe that each term in \eqref{eq:KPZ141} is a continuous functional of the $\mathscr{D}(\mathbb{T})$-valued process $\mathbf{u}^{N,1,\e}$. For example, the LHS of \eqref{eq:KPZ141} and the first three terms on the RHS are given by integrating space-time values of $\mathbf{u}^{N,1,\e}$ against deterministic functions. The last martingale term on the RHS, being determined by the remaining terms, is therefore a functional of $\mathbf{u}^{N,1,\e}$. Therefore, because $\mathbf{u}^{N,1,\e}$ converges along subsequences, any continuous functional of $\mathbf{u}^{N,1,\e}$ must also converge along subsequences. We choose the continuous functional given by the joint process consisting of the terms in \eqref{eq:KPZ141}.} We deduce by taking large-$N$ limits of \eqref{eq:KPZ141} the following identity:
\begin{align}
\langle\mathbf{u}^{\infty,1,\e}(\t,\cdot),\Phi\rangle_{\mathbb{T}} \ = \ \langle\mathbf{u}^{\infty,1,\e}(0,\cdot),\Phi\rangle_{\mathbb{T}} + \langle\mathbf{u}^{\infty,1,\e},\Delta\Phi\rangle_{\t;\mathbb{T}} - \lim_{N\to\infty}\e_{N}^{-1}\langle\Pi_{N}\mathbf{F}(\e_{N}^{1/2}\mathbf{u}^{N,1,\e}), \grad\Phi\rangle_{\t;\mathbb{T}} + \lim_{N\to\infty}\mathbf{B}^{N}(\t;\Phi). \nonumber
\end{align}
{Recall $\mathbf{B}^{N}(\t;\Phi)$ is a Brownian motion with linear-in-time quadratic variation $2\t\|\grad\Pi_{N}\Phi\|_{\x;2}^{2}$. Because $\Phi\in\mathscr{C}^{\infty}(\mathbb{T})$, we know $2\t\|\grad\Pi_{N}\Phi\|_{\x;2}^{2}\to2\t\|\grad\Phi\|_{\x;2}^{2}$ as $N\to\infty$. Therefore, $\mathbf{B}^{N}(\t;\Phi)$ must converge weakly (as a process) to a Brownian motion of quadratic variation $2\t\|\grad\Phi\|_{\x;2}^{2}$. (Technically, one needs uniform integrability of $|\mathbf{B}^{N}(\t;\Phi)|^{2}$ as $N\to\infty$ to guarantee $\mathbf{B}^{N}(\t;\Phi)$ converges to a martingale with the limit quadratic variation, but this follows from fourth-moment estimates readily available for Brownian motion.)} In \cite{GP16}, it is shown that for Brownian bridge initial data $\mathbf{h}^{N,1,\infty}$ {(that is independent of $\xi$)}, the remaining limit above converges to the process 
\begin{align}
\lim_{N\to\infty}\e_{N}^{-1}\langle\mathbf{F}(\e_{N}^{1/2}\mathbf{u}^{N,1,\infty}), \grad\Phi\rangle_{\t;\mathbb{T}} \ = \ {2^{-1}\E^{\mathrm{G}}\mathbf{F}''\cdot}\lim_{N\to\infty}\e_{N}^{-1}\langle(\mathbf{u}^{N,1,\infty})^{2},\grad\Phi\rangle_{\t;\mathbb{T}} \ \overset{\bullet}= \ \mathbf{A}(\t;\Phi). \label{eq:KPZ143}
\end{align}
{Again, since $\e_{N}^{-1}\langle\mathbf{F}(\e_{N}^{1/2}\mathbf{u}^{N,1,\infty}), \grad\Phi\rangle_{\t;\mathbb{T}}$ and $2^{-1}\E^{\mathrm{G}}\mathbf{F}''\e_{N}^{-1}\langle(\mathbf{u}^{N,1,\infty})^{2},\grad\Phi\rangle_{\t;\mathbb{T}}$ are continuous functionals of the $\mathscr{D}(\mathbb{T})$-valued process $\mathbf{u}^{N,1,\e}$, if they both converge (along subsequences) and their difference converges to zero in probability for Brownian bridge initial data $\mathbf{h}^{N,1,\infty}$, then the same must be true for $\mathbf{h}^{N,1,\e}$ initial data. This last claim follows from the path-space relative entropy estimate in Lemma \ref{lemma:KPZ12} and \eqref{eq:re21} (these show that asymptotically probability-zero events for Brownian bridge initial data $\mathbf{h}^{N,1,\infty}$ are also asymptotically probability-zero for $\mathbf{h}^{N,1,\e}$ initial data).} {Thus, we deduce, for any $\Phi\in\mathscr{C}^{\infty}(\mathbb{T})$, that}
\begin{align}
\langle\mathbf{u}^{\infty,1,\e}(\t,\cdot),\Phi\rangle_{\mathbb{T}} \ = \ \langle\mathbf{u}^{\infty,1,\e}(0,\cdot),\Phi\rangle_{\mathbb{T}} + \langle\mathbf{u}^{\infty,1,\e},\Delta\Phi\rangle_{\t;\mathbb{T}} - \mathbf{A}(\t;\Phi) + \mathbf{B}(\t;\Phi). \label{eq:KPZ144}
\end{align}
{By Theorem 2.8 in \cite{GP17}, in order to show that $\mathbf{u}^{\infty,1,\e}$ are Cole-Hopf solutions, it suffices to show that:
\begin{itemize}
\item The limit point $\mathbf{u}^{\infty,1,\e}$ of the $\mathscr{D}(\mathbb{T})$-valued processes $\mathbf{u}^{N,1,\e}$ satisfies \eqref{eq:energyestimatesolution}.
\item We have $\E|\mathrm{e}^{\langle\mathbf{u}^{\infty,1,\e}(\t,\cdot),\Theta_{\x}\rangle_{\mathbb{T}}|^{2}}<\infty$ for $\t\geq0$, where $\Theta_{\x}(\y)=\Theta(\x-\y)$ and $\Theta$ is defined by its Fourier transform $\int_{\mathbb{T}}\Theta(\y)\mathrm{e}^{-2\pi ik\y}\d\y = \mathbf{1}_{k\neq0}$. 
\end{itemize}
For these two points, note that Lemma \ref{lemma:KPZ12} actually shows the Radon-Nikodym derivative of the path-space law of $\mathbf{u}^{N,1,\e}$ with respect to the law of the solution $\mathbf{u}^{N,1,\infty}$ to \eqref{eq:UVSBE1} with stationary initial data is $\mathrm{O}_{\e}(1)$ in $\mathscr{L}^{\infty}$, not just in relative entropy. Thus, the same is true by lower semicontinuity of $\mathscr{L}^{\infty}$-norm for the Radon-Nikodym derivative of the path-space law of $\mathbf{u}^{\infty,1,\e}$ with respect to the path-space law of $\mathbf{u}^{\infty,1,\infty}$ with white noise initial data. In particular, we have the following two estimates, which reduce showing the previous two bullet points to showing the same but for $\mathbf{u}^{\infty,1,\e}$ replaced by $\mathbf{u}^{\infty,1,\infty}$:
\begin{align}
&\E|\int_{\s}^{\t}(\mathbf{u}^{\infty,1,\e}\ast\rho^{N})^{2}(-\grad\Phi)-(\mathbf{u}^{\infty,1,\e}\ast\rho^{M})^{2}(-\grad\Phi)\d r|^{2} \\
&\lesssim \ \E|\int_{\s}^{\t}(\mathbf{u}^{\infty,1,\infty}\ast\rho^{N})^{2}(-\grad\Phi)-(\mathbf{u}^{\infty,1,\infty}\ast\rho^{M})^{2}(-\grad\Phi)\d r|^{2} \label{eq:KPZ145a}\\ 
&\E|\mathrm{e}^{\langle\mathbf{u}^{\infty,1,\e}(\t,\cdot),\Theta_{\x}\rangle_{\mathbb{T}}|^{2}} \\
&\lesssim_{\e} \ \E|\mathrm{e}^{\langle\mathbf{u}^{\infty,1,\infty}(\t,\cdot),\Theta_{\x}\rangle_{\mathbb{T}}|^{2}}. \label{eq:KPZ145b}
\end{align}
The term \eqref{eq:KPZ145a} is bounded the RHS of \eqref{eq:energyestimatesolution}; this is shown in Lemma 1 of \cite{GJara}. The term \eqref{eq:KPZ145b} is finite by Theorem 2.13 in \cite{GP} (see also Proposition 2.1 in \cite{GP17}). This completes the proof.}
\end{proof}
\subsection{Estimate for $\mathbf{h}^{N,2,\e}$}
The result of this subsection is confirmation of the second bullet point in Theorem \ref{theorem:KPZ1} for our construction of $\mathbf{h}^{N,2,\e}$ made in Definition \ref{definition:KPZ11}. The proof follows by the parabolic maximum principle \cite{Evans}.
\begin{lemma}\label{lemma:KPZ15}
We first recall that $\mathbf{h}^{N,2,\e}=\mathbf{h}^{N}-\mathbf{h}^{N,1,\e}$ is defined to solve \emph{\eqref{eq:UVKPZ2}} with initial data $\mathbf{h}^{N}(0,\cdot)-\mathbf{h}^{N,1,\e}(0,\cdot)$. We have $|\mathbf{h}^{N,2,\e}(\t,\x)|\leq\e+\mathrm{o}_{N}$ for all non-negative $\t$ and $\x\in\mathbb{T}$, where $\mathrm{o}_{N}\to0$ uniformly in $\t,\x$ with probability 1.
\end{lemma}
\begin{proof}
We first prove the inequality for $\t=0$. To this end, we have $\mathbf{h}^{N,2,\e}(0,\x)=\mathbf{h}^{N}(0,\x)-\mathbf{h}^{N,1,\e}(0,\x)$, so
\begin{align}
|\mathbf{h}^{N,2,\e}(0,\x)| \ \leq \ |\mathbf{h}^{N}(0,\x)-\mathbf{h}(0,\x)|+|\mathbf{h}(0,\x)-\mathbf{h}^{1,\e}(0,\x)| + |\mathbf{h}^{1,\e}(0,\x)-\mathbf{h}^{N,1,\e}(0,\x)|.
\end{align}
The first term on the far RHS is $\mathrm{o}_{N}$ by assumption. Moreover, by construction in Definition \ref{definition:KPZ11}, the second term on the far RHS is defined to be at most $\e$. The last term is $\mathrm{o}_{N}$, which vanishes uniformly in $\mathbb{T}$ with probability 1 by classical Brownian motion estimates. This gives the proposed estimate for the initial time. It now suffices to get the following for all $\t\geq0$ and $\x\in\mathbb{T}$ simultaneously with probability 1:
\begin{align}
|\mathbf{h}^{N,2,\e}(\t,\x)| \ \leq \ {\sup}_{\y\in\mathbb{T}}|\mathbf{h}^{N,2,\e}(0,\y)|. \label{eq:KPZ151}
\end{align}
Indeed, this bound would propagate the validity of the proposed inequality at $\t=0$ globally in time with probability 1. To prove this last bound, we first note {that $\mathbf{h}^{N,1,\e}$ is smooth with probability 1 by Lemma \ref{lemma:pde2}}. {We will only use this fact to guarantee that all derivatives of $\mathbf{h}^{N,1,\e}$ are defined pointwise with probability 1. Now consider the following.
\begin{itemize}
\item Fix $\chi\in\mathscr{C}^{\infty}(\R)$ such that its support is contained in $[-1,1]$, it is non-negative-valued, and $\int_{\R}\chi(\x)\d\x=1$. For any $\upsilon>0$, define $\chi_{\upsilon}(\x)=\upsilon^{-1}\chi(\x/\upsilon)$ for $\x\in\R$. 
\item Define $\mathbf{F}_{\upsilon}=\mathbf{F}\star\chi_{\upsilon}$ to be a smoothing of $\mathbf{F}$. Define $\mathbf{h}^{N,2,\e,\upsilon}$ to be the solution to the following PDE on $\R_{\geq0}\times\mathbb{T}$ with the same initial data $\mathbf{h}^{N,2,\e,\upsilon}(0,\x)=\mathbf{h}^{N,2,\e}(0,\x)$:
\begin{align}
\partial_{\t}\mathbf{h}^{N,2,\e,\upsilon} \ = \ \Delta\mathbf{h}^{N,2,\e,\upsilon} + \e_{N}^{-1}\mathbf{F}_{\upsilon}(\e_{N}^{1/2}\grad(\mathbf{h}^{N,1,\e}+\mathbf{h}^{N,2,\e,\upsilon})) - \e_{N}^{-1}\mathbf{F}_{\upsilon}(\e_{N}^{1/2}\grad\mathbf{h}^{N,1,\e}). \label{eq:KPZ152}
\end{align}
\end{itemize}
Let us comment briefly on the above construction; this will be key to our proof of Lemma \ref{lemma:KPZ15}.
\begin{itemize}
\item Note the support of $\chi_{\upsilon}$ is contained in $[-\upsilon,\upsilon]$.
\item Because $\mathbf{F}:\R\to\R$ is uniformly Lipschitz and $\chi_{\upsilon}$ is supported in $[-\upsilon,\upsilon]$, we know 
\begin{align}
|\mathbf{F}(\x)-\mathbf{F}_{\upsilon}(\x)|=|\int_{\R}\chi_{\upsilon}(\x-\y)(\mathbf{F}(\x)-\mathbf{F}(\y))\d\y| \ \lesssim_{\mathbf{F}} \ \upsilon;
\end{align}
the dependence on $\mathbf{F}$ is through its Lipschitz constant. The bound above is independent of $\x\in\R$.
\item The Lipschitz norm of $\grad^{k}\mathbf{F}_{\upsilon}$ is bounded by $\mathrm{O}_{k,\upsilon}(1)$ times that of $\mathbf{F}$ for all $\upsilon>0$. Indeed, for any $\x_{1},\x_{2}$, we have
\begin{align}
|\grad^{k}\mathbf{F}_{\upsilon}(\x_{1})-\grad^{k}\mathbf{F}_{\upsilon}(\x_{2})| \ &= \ |\int_{\R}\left(\grad^{k}\chi_{\upsilon}(\x_{1}-\y)-\grad^{k}\chi_{\upsilon}(\x_{2}-\y)\right)\mathbf{F}(\y)\d\y| \\
&= \ |\int_{\R}\grad^{k}\chi_{\upsilon}(\y)\left(\mathbf{F}(\x_{1}-\y)-\mathbf{F}(\x_{2}-\y)\right)\d\y| \\
&\lesssim_{k,\upsilon} \ \|\mathbf{F}\|_{\mathrm{Lip}}|\x_{1}-\x_{2}|,
\end{align}
where $\|\mathbf{F}\|_{\mathrm{Lip}}$ denotes the Lipschitz constant of $\mathbf{F}$.
\end{itemize}
To prove \eqref{eq:KPZ151}, it suffices to prove two ingredients. The first is that $\mathbf{h}^{N,2,\e,\upsilon}\to_{\upsilon\to0}\mathbf{h}^{N,2,\e}$ locally uniformly in time and uniformly in space (with probability 1). The second is that \eqref{eq:KPZ151} holds if we replace $\mathbf{h}^{N,2,\e}$ on both sides of the inequality by $\mathbf{h}^{N,2,\e,\upsilon}$. Note that the RHS does not change if we replace $\mathbf{h}^{N,2,\e}$ therein by $\mathbf{h}^{N,2,\e,\upsilon}$ as these two have the same initial data by construction. The second ingredient is almost immediate. Again, by Lemma \ref{lemma:pde3}, \eqref{eq:KPZ152} admits a classical solution. Now, let us rewrite \eqref{eq:KPZ152} as
\begin{align}
\partial_{\t}\mathbf{h}^{N,2,\e,\upsilon} \ = \ \Delta\mathbf{h}^{N,2,\e,\upsilon} + \e_{N}^{-1}\left(\frac{\mathbf{F}_{\upsilon}(\e_{N}^{1/2}\grad(\mathbf{h}^{N,1,\e}+\mathbf{h}^{N,2,\e,\upsilon})) - \e_{N}^{-1}\mathbf{F}_{\upsilon}(\e_{N}^{1/2}\grad\mathbf{h}^{N,1,\e})}{\grad\mathbf{h}^{N,2,\e,\upsilon}}\right)\grad\mathbf{h}^{N,2,\e,\upsilon}. \label{eq:KPZ153}
\end{align}
The coefficient for $\grad\mathbf{h}^{N,2,\e,\upsilon}$ on the RHS of \eqref{eq:KPZ153} is uniformly bounded in $\x$ (for each $N$):
\begin{align}
|\e_{N}^{-1}\frac{\mathbf{F}_{\upsilon}(\e_{N}^{1/2}\grad(\mathbf{h}^{N,1}+\mathbf{h}^{N,2,\e,\upsilon})) - \e_{N}^{-1}\mathbf{F}_{\upsilon}(\e_{N}^{1/2}\grad\mathbf{h}^{N,1})}{\grad\mathbf{h}^{N,2,\e,\upsilon}}| \ \lesssim_{N,\mathbf{F}} \ |\frac{\grad\mathbf{h}^{N,1,\e}+\grad\mathbf{h}^{N,2,\e,\upsilon}-\grad\mathbf{h}^{N,1,\e}}{\grad\mathbf{h}^{N,2,\e,\upsilon}}| \ = \ 1.
\end{align}
We note the dependence of $\mathbf{F}$ above is through its Lipschitz constant. Thus, to \eqref{eq:KPZ153}, we can apply the comparison principle to deduce \eqref{eq:KPZ151} holds for $\mathbf{h}^{N,2,\e,\upsilon}$ in place of $\mathbf{h}^{N,2,\e}$. (For the comparison principle, see Theorem 8 of Chapter 7.1.4 in \cite{Evans} for $n=1$ and $a^{11}=1$ and $b^{1}=\e_{N}^{-1}\frac{\mathbf{F}_{\upsilon}(\e_{N}^{1/2}\grad(\mathbf{h}^{N,1}+\mathbf{h}^{N,2,\e,\upsilon})) - \e_{N}^{-1}\mathbf{F}_{\upsilon}(\e_{N}^{1/2}\grad\mathbf{h}^{N,1})}{\grad\mathbf{h}^{N,2,\e,\upsilon}}$ and $c=0$. (We note that although Theorem 8 of Chapter 7.1.4 in \cite{Evans} assumes continuity of coefficients, the proof works for boundedness of the first-order coefficient.)

It now remains to prove the convergence $\mathbf{h}^{N,2,\e,\upsilon}\to_{\upsilon\to0}\mathbf{h}^{N,2,\e}$ locally uniformly in time and uniformly in space (with probability 1). To this end, we start with their mild representations. In particular, starting with the definition of $\mathbf{h}^{N,2,\e}$ and Lemma \ref{lemma:pde3}, we have the following calculation:
\begin{align}
&\mathbf{h}^{N,2,\e}(\t,\x)-\mathbf{h}^{N,2,\e,\upsilon}(\t,\x) \\
&= \ \int_{0}^{\t}\mathrm{e}^{(\t-\s)\Delta}\left(\e_{N}^{-1}\mathbf{F}(\e_{N}^{1/2}\grad(\mathbf{h}^{N,1,\e}(\s,\cdot)-\mathbf{h}^{N,2,\e}(\s,\cdot)))-\right)(\x)\d\s \\
&\quad-\int_{0}^{\t}\mathrm{e}^{(\t-\s)\Delta}\left(\e_{N}^{-1}\mathbf{F}_{\upsilon}(\e_{N}^{1/2}\grad(\mathbf{h}^{N,1,\e}(\s,\cdot)-\mathbf{h}^{N,2,\e,\upsilon}(\s,\cdot)))-\right)(\x)\d\s \\
&= \ \int_{0}^{\t}\mathrm{e}^{(\t-\s)\Delta}\left(\e_{N}^{-1}\mathbf{F}(\e_{N}^{1/2}\grad(\mathbf{h}^{N,1,\e}(\s,\cdot)-\mathbf{h}^{N,2,\e}(\s,\cdot)))-\right)(\x)\d\s \label{eq:KPZ154a} \\
&\quad-\int_{0}^{\t}\mathrm{e}^{(\t-\s)\Delta}\left(\e_{N}^{-1}\mathbf{F}(\e_{N}^{1/2}\grad(\mathbf{h}^{N,1,\e}(\s,\cdot)-\mathbf{h}^{N,2,\e,\upsilon}(\s,\cdot)))-\right)(\x)\d\s\label{eq:KPZ154b} \\
&\quad+\int_{0}^{\t}\mathrm{e}^{(\t-\s)\Delta}\left(\e_{N}^{-1}\mathbf{F}(\e_{N}^{1/2}\grad(\mathbf{h}^{N,1,\e}(\s,\cdot)-\mathbf{h}^{N,2,\e,\upsilon}(\s,\cdot)))-\right)(\x)\d\s\label{eq:KPZ154c} \\
&\quad-\int_{0}^{\t}\mathrm{e}^{(\t-\s)\Delta}\left(\e_{N}^{-1}\mathbf{F}_{\upsilon}(\e_{N}^{1/2}\grad(\mathbf{h}^{N,1,\e}(\s,\cdot)-\mathbf{h}^{N,2,\e,\upsilon}(\s,\cdot)))-\right)(\x)\d\s.\label{eq:KPZ154d}
\end{align}
Let us first estimate the contribution of \eqref{eq:KPZ154a}-\eqref{eq:KPZ154b}. Because $\mathbf{F}$ is uniformly Lipschitz, we have 
\begin{align}
&|\int_{0}^{\t}\mathrm{e}^{(\t-\s)\Delta}\left(\e_{N}^{-1}\mathbf{F}(\e_{N}^{1/2}\grad(\mathbf{h}^{N,1,\e}(\s,\cdot)-\mathbf{h}^{N,2,\e}(\s,\cdot)))-\right)(\x)\d\s \\
&-\int_{0}^{\t}\mathrm{e}^{(\t-\s)\Delta}\left(\e_{N}^{-1}\mathbf{F}(\e_{N}^{1/2}\grad(\mathbf{h}^{N,1,\e}(\s,\cdot)-\mathbf{h}^{N,2,\e,\upsilon}(\s,\cdot)))-\right)(\x)| \\
&\lesssim_{\mathbf{F},N} \ \int_{0}^{\t}\mathrm{e}^{(\t-\s)\Delta}\left(|\grad\mathbf{h}^{N,2,\e}(\s,\cdot)-\grad\mathbf{h}^{N,2,\e,\upsilon}(\s,\cdot)|\right)(\x)\d\s \\
&\leq \ \t\sup_{0\leq\s\leq\t}\sup_{\y\in\mathbb{T}}|\grad\mathbf{h}^{N,2,\e}(\s,\y)-\grad\mathbf{h}^{N,2,\e,\upsilon}(\s,\y)|.
\end{align}
The last estimate follows by contractivity of convolution with the heat kernel on $\mathbb{T}$ (with respect to the $\mathscr{L}^{\infty}$-norm on $\mathbb{T}$). As for \eqref{eq:KPZ154c}-\eqref{eq:KPZ154d}, because $|\mathbf{F}-\mathbf{F}_{\upsilon}|\lesssim\upsilon$ uniformly, we have 
\begin{align}
&|\int_{0}^{\t}\mathrm{e}^{(\t-\s)\Delta}\left(\e_{N}^{-1}\mathbf{F}(\e_{N}^{1/2}\grad(\mathbf{h}^{N,1,\e}(\s,\cdot)-\mathbf{h}^{N,2,\e,\upsilon}(\s,\cdot)))-\right)(\x)\d\s \\
&\quad-\int_{0}^{\t}\mathrm{e}^{(\t-\s)\Delta}\left(\e_{N}^{-1}\mathbf{F}_{\upsilon}(\e_{N}^{1/2}\grad(\mathbf{h}^{N,1,\e}(\s,\cdot)-\mathbf{h}^{N,2,\e,\upsilon}(\s,\cdot)))-\right)(\x)\d\s| \\
&\lesssim_{N} \ \sup_{\y\in\R}|\mathbf{F}(\y)-\mathbf{F}_{\upsilon}(\y)|\times\int_{0}^{\t}\mathrm{e}^{(\t-\s)\Delta}(1)\d\s \\
&\lesssim \ \t\upsilon.
\end{align}
Combining the previous three displays gives the following estimate for any $T\geq0$ fixed:
\begin{align}
\sup_{0\leq\t\leq T}\sup_{\x\in\mathbb{T}}|\mathbf{h}^{N,2,\e}(\t,\x)-\mathbf{h}^{N,2,\e,\upsilon}(\t,\x)| \ \lesssim_{N,\mathbf{F}} \ T\sup_{0\leq\s\leq\t}\sup_{\y\in\mathbb{T}}|\grad\mathbf{h}^{N,2,\e}(\s,\y)-\grad\mathbf{h}^{N,2,\e,\upsilon}(\s,\y)| + T\upsilon. \label{eq:KPZ155}
\end{align}
Thus, it remains to estimate the double supremum on the RHS of \eqref{eq:KPZ155}. To this end, we proceed in very similar fashion, namely estimating gradients of \eqref{eq:KPZ154a}-\eqref{eq:KPZ154d}. First, let us denote by $\mathbf{G}_{\t}(\x)$ the heat kernel for $\mathrm{e}^{\t\Delta}$ on the torus. Taking gradients of \eqref{eq:KPZ154a}-\eqref{eq:KPZ154b} (with respect to $\x$), we have
\begin{align}
&|\int_{0}^{\t}\grad\mathrm{e}^{(\t-\s)\Delta}\left(\e_{N}^{-1}\mathbf{F}(\e_{N}^{1/2}\grad(\mathbf{h}^{N,1,\e}(\s,\cdot)-\mathbf{h}^{N,2,\e}(\s,\cdot)))-\right)(\x)\d\s \\
&-\int_{0}^{\t}\grad\mathrm{e}^{(\t-\s)\Delta}\left(\e_{N}^{-1}\mathbf{F}(\e_{N}^{1/2}\grad(\mathbf{h}^{N,1,\e}(\s,\cdot)-\mathbf{h}^{N,2,\e,\upsilon}(\s,\cdot)))-\right)(\x)| \\
&\lesssim_{\mathbf{F},N} \ \int_{0}^{\t}\int_{\mathbb{T}}|\grad\mathbf{G}_{\t-\s}(\x-\y)|\cdot|\grad\mathbf{h}^{N,2,\e}(\s,\y)-\grad\mathbf{h}^{N,2,\e,\upsilon}(\s,\y)|\d\y\d\s \\
&\lesssim \ \int_{0}^{\t}(1+|\t-\s|^{-\frac12})\sup_{\y\in\mathbb{T}}|\grad\mathbf{h}^{N,2,\e}(\s,\y)-\grad\mathbf{h}^{N,2,\e,\upsilon}(\s,\y)|\d\s,
\end{align}
where the last estimate follows by $\int_{\mathbb{T}}|\grad\mathbf{G}_{\t-\s}(\x-\y)\d\y|\lesssim1+|\t-\s|^{-1/2}$; see Lemma \ref{lemma:pde3}. Taking gradients of \eqref{eq:KPZ154c}-\eqref{eq:KPZ154d}, we have 
\begin{align}
&|\int_{0}^{\t}\grad\mathrm{e}^{(\t-\s)\Delta}\left(\e_{N}^{-1}\mathbf{F}(\e_{N}^{1/2}\grad(\mathbf{h}^{N,1,\e}(\s,\cdot)-\mathbf{h}^{N,2,\e,\upsilon}(\s,\cdot)))-\right)(\x)\d\s \\
&\quad-\int_{0}^{\t}\grad\mathrm{e}^{(\t-\s)\Delta}\left(\e_{N}^{-1}\mathbf{F}_{\upsilon}(\e_{N}^{1/2}\grad(\mathbf{h}^{N,1,\e}(\s,\cdot)-\mathbf{h}^{N,2,\e,\upsilon}(\s,\cdot)))-\right)(\x)\d\s| \\
&\lesssim_{N} \ \sup_{\y\in\R}|\mathbf{F}(\y)-\mathbf{F}_{\upsilon}(\y)|\times\int_{0}^{\t}\int_{\mathbb{T}}|\grad\mathbf{G}_{\t-\s}(\x-\y)|\d\y\d\s \\
&\lesssim \ \sup_{\y\in\R}|\mathbf{F}(\y)-\mathbf{F}_{\upsilon}(\y)|\times\int_{0}^{\t}1+|\t-\s|^{-\frac12}\d\s \\
&\lesssim \ (\t+\t^{\frac12})\upsilon.
\end{align}
Combining the previous two displays, we deduce the following for any $\t\geq0$:
\begin{align}
&\sup_{\x\in\mathbb{T}}|\grad\mathbf{h}^{N,2,\e}(\t,\x)-\grad\mathbf{h}^{N,2,\e,\upsilon}(\t,\x)| \label{eq:KPZ156a}\\
&\lesssim_{N,\mathbf{F}} \ \int_{0}^{\t}(1+|\t-\s|^{-\frac12})\sup_{\y\in\mathbb{T}}|\grad\mathbf{h}^{N,2,\e}(\s,\y)-\grad\mathbf{h}^{N,2,\e,\upsilon}(\s,\y)|\d\s+ (\t+\t^{\frac12})\upsilon. \label{eq:KPZ156b}
\end{align}
We can now apply the Gronwall inequality to deduce
\begin{align}
\sup_{\x\in\mathbb{T}}|\grad\mathbf{h}^{N,2,\e}(\t,\x)-\grad\mathbf{h}^{N,2,\e,\upsilon}(\t,\x)| \ &\lesssim \ (\t+\t^{\frac12})\upsilon\times\exp\left(\int_{0}^{\t}1+|\t-\s|^{-\frac12}\d\s\right) \\
&\lesssim \ (\t+\t^{\frac12})\exp(\t+\t^{\frac12})\upsilon.
\end{align}
Combining this last display with \eqref{eq:KPZ155}, we obtain
\begin{align}
\sup_{0\leq\t\leq T}\sup_{\x\in\mathbb{T}}|\mathbf{h}^{N,2,\e}(\t,\x)-\mathbf{h}^{N,2,\e,\upsilon}(\t,\x)| \ \lesssim_{N,\mathbf{F}} \ T(T^{\frac12}+T)\mathrm{e}^{T^{1/2}+T}\upsilon + T\upsilon,
\end{align}
which shows, for each fixed $N>0$, the convergence $\mathbf{h}^{N,2,\e,\upsilon}\to_{\upsilon\to0}\mathbf{h}^{N,2,\e}$ locally uniformly in time and uniformly in space (with probability 1). This completes the proof.}
\end{proof}
\subsection{Proof of Theorem \ref{theorem:KPZ1}}
According to Lemma \ref{lemma:KPZ14} and Lemma \ref{lemma:KPZ15}, the two bullet points in Theorem \ref{theorem:KPZ1} are satisfied for our particular choice of $\mathbf{h}^{N,1,\e}$ and $\mathbf{h}^{N,2,\e}$ given by choosing their initial data via Definition \ref{definition:KPZ11} and requiring them to solve \eqref{eq:UVKPZ1} and \eqref{eq:UVKPZ2}, respectively. To compute the large-$N$ limit of $\mathbf{u}^{N}$, similar to the proof of Lemma \ref{lemma:KPZ14}, it suffices to show that the one-dimensional process $\mathbf{u}^{N}(\t,\Phi)\overset{\bullet}=\langle\mathbf{u}^{N}(\t,\cdot),\Phi\rangle_{\mathbb{T}}$ converges to the action of the solution to \eqref{eq:SBE} on $\Phi$ provided any generic smooth function $\Phi\in\mathscr{C}^{\infty}(\mathbb{T})$; this convergence must be in $\mathscr{C}(\R_{\geq0},\R)$ with the locally uniform topology, and it will be directly inherited from convergence of $\mathbf{u}^{N,1,\e}$ and estimates for $\mathbf{u}^{N,2,\e}$ from Lemma \ref{lemma:KPZ15}. Let us clarify that convergence in locally uniform topology on $\mathscr{C}(\R_{\geq0},\R)$ means convergence uniformly in $\mathscr{C}([0,\t_{+}],\R)$ for every $\t_{+}\geq0$. By integration-by-parts, we have
\begin{align}
\mathbf{u}^{N}(\t,\Phi) \ = \ \langle\mathbf{u}^{N,1,\e}(\t,\cdot),\Phi\rangle_{\mathbb{T}} - \langle\mathbf{h}^{N,2,\e}(\t,\cdot),\grad\Phi\rangle_{\mathbb{T}}. \label{eq:KPZ1I}
\end{align}
The second term on the RHS of \eqref{eq:KPZ1I} is bounded by order $\e$ times $\Phi$-data; this follows by the height function estimate in Lemma \ref{lemma:KPZ15}. Because this comparison holds uniformly on finite time-horizons, for any $\t_{+}\geq0$, uniformly in $\t\in[0,\t_{+}]$, we have the following in which $\mathbf{u}(\t,\Phi)$ denotes the action of the Cole-Hopf solution to \eqref{eq:SBE} on $\Phi$:
\begin{align}
|\mathbf{u}^{N}(\t,\Phi)-\mathbf{u}(\t,\Phi)| \ &\leq \ |\langle\mathbf{u}^{N,1,\e}(\t,\cdot),\Phi\rangle_{\mathbb{T}}-\mathbf{u}(\t,\Phi)| + \kappa(\Phi)\e + \kappa(\Phi)\mathrm{o}_{N}. \label{eq:KPZ1II}
\end{align}
Above, the constant $\kappa(\Phi)$ depends only on its argument $\Phi$. Observe the LHS of \eqref{eq:KPZ1II} does not depend on $\e$. Therefore, if we can show the first term on the RHS is bounded by $\kappa(\Phi)\e+ \kappa(\Phi)\mathrm{o}_{N}$, then because $\e$ on the RHS of \eqref{eq:KPZ1II} is arbitrary, we would be able to deduce $\mathbf{u}^{N}(\t,\Phi)\to\mathbf{u}(\t,\Phi)$ uniformly on $[0,\t_{+}]$ in probability provided any finite $\t_{+}\geq0$ and $\Phi$. As noted prior to \eqref{eq:KPZ1I}, this would complete the proof of convergence. We first note that by Lemma \ref{lemma:KPZ12}, it suffices to replace $\mathbf{u}^{N,1,\e}$ in \eqref{eq:KPZ1II} with the Cole-Hopf solution $\mathbf{u}^{1,\e}$ to \eqref{eq:SBE} with initial data the weak derivative of Brownian bridge {independent of $\xi$ and} conditioned to be within $\e$ of $\mathbf{h}(0,\cdot)$. It then suffices to employ the following estimate that compares solutions to \eqref{eq:SBE} with initial data whose antiderivatives are very close uniformly on $\mathbb{T}$; precisely, we choose $\alpha=\e$ for the result below, and we choose $\mathbf{h}^{1}(0,\cdot)=\mathbf{h}^{1,\e}(0,\cdot)$ from Definition \ref{definition:KPZ11} and $\mathbf{h}^{2}(0,\cdot)=\mathbf{h}(0,\cdot)$:
\begin{lemma}\label{lemma:KPZ16}
Suppose $\mathbf{u}^{1}$ and $\mathbf{u}^{2}$ are Cole-Hopf solutions to \emph{\eqref{eq:SBE}} with initial data $\mathbf{u}^{1}=\grad\mathbf{h}^{1}(0,\cdot)$ and $\mathbf{u}^{2}=\grad\mathbf{h}^{2}(0,\cdot)$, in which {$\mathbf{h}^{1}(0,\cdot)$ and $\mathbf{h}^{2}(0,\cdot)$} are continuous on $\mathbb{T}$ and satisfy the following uniformly in $\x\in\mathbb{T}$ with probability 1:
\begin{align}
|\mathbf{h}^{1}(0,\x)-\mathbf{h}^{2}(0,\x)| \ \leq \ \alpha.
\end{align}
Provided any $\Phi\in\mathscr{C}^{\infty}(\mathbb{T})$, we have the following uniformly in $\t\geq0$ with probability 1, where $\kappa(\Phi)$ depends only on $\Phi$:
\begin{align}
|\langle\mathbf{u}^{1}(\t,\cdot),\Phi\rangle_{\mathbb{T}}-\langle\mathbf{u}^{2}(\t,\cdot),\Phi\rangle_{\mathbb{T}}| \ \leq \ \kappa(\Phi)\alpha. \label{eq:KPZ16II}
\end{align}
\end{lemma}
\begin{proof}
We first rewrite the LHS of \eqref{eq:KPZ16II} by realizing $\mathbf{u}^{i}=\grad\mathbf{h}^{i}$, where $\mathbf{h}^{i}$ is the Cole-Hopf solution to \eqref{eq:KPZ}. We then move the gradient onto $\Phi$ and write $\mathbf{h}^{i}=\beta^{-1}\log\mathbf{Z}^{i}$, where $\mathbf{Z}^{i}$ solves the \emph{SHE} $\partial_{\t}\mathbf{Z}^{i}=\Delta\mathbf{Z}^{i}+\beta\mathbf{Z}^{i}\xi$. This gives
\begin{align}
\langle\mathbf{u}^{1}(\t,\cdot),\Phi\rangle_{\mathbb{T}}-\langle\mathbf{u}^{2}(\t,\cdot),\Phi\rangle_{\mathbb{T}} \ = \ -\beta^{-1}\left(\langle\log\mathbf{Z}^{1}(\t,\cdot),\grad\Phi\rangle_{\mathbb{T}}-\langle\log\mathbf{Z}^{2}(\t,\cdot),\grad\Phi\rangle_{\mathbb{T}}\right). \label{eq:KPZ16II1}
\end{align}
Because $\mathbf{h}^{i}$ are continuous and satisfy an a priori global estimate for the initial data, it suffices to compare the solutions SHE with globally close initial data. This follows by the comparison principle for the SHE equations with initial data $\exp(\beta(\mathbf{h}^{2}(0,\x)\pm\alpha))$ and $\exp(\beta\mathbf{h}^{1}(0,\x))$. In particular, we know $\exp(\beta\mathbf{h}^{2}(0,\x)-\beta\alpha)\leq\exp(\beta\mathbf{h}^{1}(0,\x))\leq\exp(\beta\mathbf{h}^{2}(0,\x)+\beta\alpha)$ given our initial data assumption since $\beta$ is positive. If we run the SHE with these three initial data, the comparison principle \cite{Mu} for the SHE implies this two-sided bound remains true not just for time 0 but for any positive time $\t\geq0$ because $\exp(\pm\beta\alpha)$ are constants and SHE is linear. Because exponential and logarithm functions preserve order, we deduce $\mathbf{h}^{2}(\t,\x)-\alpha\leq\mathbf{h}^{1}(\t,\x)\leq\mathbf{h}^{2}(\t,\x)+\alpha$ for all $\t\geq0$ and $\x\in\mathbb{T}$. Therefore, the RHS of \eqref{eq:KPZ16II1} is controlled by the supremum of $|\grad\Phi|$ times $\alpha|\mathbb{T}|=\alpha$ as $\mathbb{T}$ is unit-length, so we are done.
\end{proof}
%
%
%
\section{Proof of Theorem \ref{theorem:KPZ2}}
For the sake of clarity, we will establish all the notation to be used in this section in the following list. In particular, we will not inherit any notation used in the previous sections unless explicitly mentioned in order to reset notation and avoid any possible confusion in presentation.
\begin{itemize}
\item Consider a deterministic continuous function $\mathbf{h}(0,\cdot)$ thought of as initial data in Theorem \ref{theorem:KPZ2}. We define $\mathbf{h}(\t,\x)$ as the Cole-Hopf solution to \eqref{eq:KPZ} with initial data $\mathbf{h}(0,\cdot)$ and $\mathbf{u}(\t,\cdot)=\grad\mathbf{h}(\t,\cdot)$ the Cole-Hopf solution to \eqref{eq:SBE}. 
\item Provided any positive $\e$, we define $\mathbf{h}^{\e}(0,\cdot)$ to be Brownian bridge {independent of $\xi$ and} conditioned to be within $\e$ of $\mathbf{h}(0,\cdot)$ uniformly on $\mathbb{T}$. We define $\mathbf{h}^{\e}(\t,\x)$ as the Cole-Hopf solution to \eqref{eq:KPZ} with initial data $\mathbf{h}^{\e}(0,\cdot)$ and $\mathbf{u}^{\e}(\t,\cdot)=\grad\mathbf{h}^{\e}(\t,\cdot)$ the Cole-Hopf solution to \eqref{eq:SBE}. The noise in this and the previous bullet points are the same noise in a pathwise sense.
\item Lastly, we define $\wt{\mathbf{h}}^{\e}$ to be difference $\mathbf{h}-\mathbf{h}^{\e}$. This is a non-regularized analog of $\mathbf{h}^{N,2,\e}$ from Theorem \ref{theorem:KPZ1}.
\end{itemize}
Theorem \ref{theorem:KPZ2} will follow from the list of estimates provided below, as we illustrate after we present these estimates. Let us first note that we will use some notation not-yet-established, so as to match it soon with notation in Theorem \ref{theorem:KPZ2}.
\begin{lemma}\label{lemma:KPZ21}
We have $|\wt{\mathbf{h}}^{\e}(\t,\x)|\leq\e$ for all $\t\geq0$ and $\x\in\mathbb{T}$ with probability 1.
\end{lemma}
\begin{lemma}\label{lemma:KPZ22}
Let $\mathbf{Q}^{\e}$ denote the law of $\mathbf{h}^{\e}$ as a probability measure on $\mathscr{C}(\mathbb{T})$. Recalling $\mathbf{P}^{\infty}$ as the law of Brownian bridge on $\mathscr{C}(\mathbb{T})$, we have the relative entropy estimate $\mathsf{H}(\mathbf{Q}^{\e}|\mathbf{P}^{\infty})\leq\kappa(\e,\mathbf{h})$ for all $\e$, where $\kappa(\e,\mathbf{h})$ depends only on $\e$ and the modulus of continuity of the initial data $\mathbf{h}(0,\cdot)$. Therefore, by \emph{Lemma \ref{lemma:re2}}, the same relative entropy estimate holds if we replace $(\mathbf{Q}^{\e},\mathbf{P}^{\infty})$ by their respective pushforward under any common map.
\end{lemma}
\begin{lemma}\label{lemma:KPZ23}
We first let $\grad\mathbf{Q}^{\e}(\t)$ denote the law of $\mathbf{u}^{\e}(\t,\cdot)=\grad\mathbf{h}^{\e}(\t,\cdot)$ as a probability measure on $\mathscr{D}(\mathbb{T})$, assuming the initial measure is given by $\grad\mathbf{Q}^{\e}$, the pushforward of $\mathbf{Q}^{\e}$ from \emph{Lemma \ref{lemma:KPZ22}} under the gradient map $\mathscr{C}(\mathbb{T})\to\mathscr{D}(\mathbb{T})$. Recalling from \emph{Theorem \ref{theorem:KPZ2}} that $\eta$ denotes the Gaussian white noise measure on $\mathscr{D}(\mathbb{T})$, then for any $\t\geq0$, we have $\mathsf{H}(\grad\mathbf{Q}^{\e}(\t)|\eta)\leq\exp(-C\t)\mathsf{H}(\grad\mathbf{Q}^{\e}|\eta)$, in which $C$ is a positive constant that depends on nothing.
\end{lemma}
Let us now prove Theorem \ref{theorem:KPZ2}. We will choose $\mathbf{L}^{\e}(\t)$ therein to be $\grad\mathbf{Q}^{\e}(\t)$ constructed in Lemma \ref{lemma:KPZ23}. By Lemma \ref{lemma:KPZ21}, the Wasserstein estimate needed in Theorem \ref{theorem:KPZ2} follows immediately; the coupling between $\mathbf{L}(\t)$ and $\mathbf{L}^{\e}(\t)=\grad\mathbf{Q}^{\e}(\t)$ used in the Wasserstein metric is the coupling of $\mathbf{u}(\t,\cdot)$ and $\mathbf{u}^{\e}(\t,\cdot)$ used in the bullet points prior to Lemma \ref{lemma:KPZ21}, namely coupling the noises that define $\mathbf{u}(\t,\cdot)$ and $\mathbf{u}^{\e}(\t,\cdot)$. Indeed, Lemma \ref{lemma:KPZ21} implies that for this coupling, $\mathbf{u}(\t,\cdot)$ and $\mathbf{u}^{\e}(\t,\cdot)$ are within $\e$ of each other with respect to the metric used in the Wasserstein distance; below, we adopt the integration notation introduced and used in the proof of Lemma \ref{lemma:KPZ14}:
\begin{align}
|\langle\mathbf{u}(\t,\cdot),\Phi\rangle_{\mathbb{T}}-\langle\mathbf{u}^{\e}(\t,\cdot),\Phi\rangle_{\mathbb{T}}| \ = \ |\langle\mathbf{h}(\t,\cdot),\grad\Phi\rangle_{\mathbb{T}}-\langle\mathbf{h}^{\e}(\t,\cdot),\grad\Phi\rangle_{\mathbb{T}}| \ \leq \ \langle|\wt{\mathbf{h}}^{\e}(\t,\cdot)|,|\grad\Phi|\rangle_{\mathbb{T}} \ \leq \ \e\sup_{\x\in\mathbb{T}}|\grad\Phi(\x)|. \nonumber
\end{align}
The relative entropy estimate in Theorem \ref{theorem:KPZ2} follows from combining Lemma \ref{lemma:KPZ22} and Lemma \ref{lemma:KPZ23} with the observation that $\eta=\grad\mathbf{P}^{\infty}$ as probability measures, where $\grad\mathbf{P}^{\infty}$ denotes the pushforward of the Brownian bridge measure $\mathbf{P}^{\infty}$ on $\mathscr{C}(\mathbb{T})\subseteq\mathscr{D}(\mathbb{T})$ under the gradient map. This completes the proof. \qed
\begin{proof}[Proof of \emph{Lemma \ref{lemma:KPZ21}}]
This follows from the proof of Lemma \ref{lemma:KPZ16} for $\mathbf{h}^{1}=\mathbf{h}(0,\cdot)$ and $\mathbf{h}^{2}=\mathbf{h}^{\e}(0,\cdot)$ and $\alpha=\e$.
\end{proof}
\begin{proof}[Proof of \emph{Lemma \ref{lemma:KPZ22}}]
As in the proof of Lemma \ref{lemma:KPZ12}, it amounts to estimate from below the probability that Brownian bridge is within $\e$ of a deterministic continuous function on $\mathbb{T}$; said estimate can only depend on $\e$ and properties of said deterministic continuous function. But this was estimated in the proof of Lemma \ref{lemma:KPZ12}, so we are done.
\end{proof}
\begin{proof}[Proof of \emph{Lemma \ref{lemma:KPZ23}}]
Define $\mathbf{u}^{N,\e,1}$ as the solution to \eqref{eq:UVSBE1} for $\mathbf{F}(\x)=\x^2$ with initial data given by $\Pi_{N}\mathbf{u}^{\e}$. We then have $\lim_{N\to\infty}\mathbf{u}^{N,\e,1}=\mathbf{u}^{\e}$ as measures on $\mathscr{C}(\R_{\geq0},\mathscr{D}(\mathbb{T}))$ because the relative entropy of $\mathbf{u}^{\e}$ with respect to white noise $\eta$ measure is uniformly bounded in $N$ by Lemma \ref{lemma:KPZ22}; this convergence follows as in the proof of Lemma \ref{lemma:KPZ14}. Because the relative entropy is lower semi-continuous in both of its arguments jointly (see Appendix 1.8 in \cite{KL}), we then deduce the following relative entropy estimate in which $\grad\mathbf{Q}^{\e,N}(\t)$ is the law of $\mathbf{u}^{N,\e,1}(\t,\cdot)$ and $(\Pi_{N})_{\ast}\eta$ is the Fourier smoothing of white noise measure $\eta$:
\begin{align}
\mathsf{H}\left(\grad\mathbf{Q}^{\e}(\t)|\eta\right) \ \leq \ \limsup_{N\to\infty}\mathsf{H}\left(\grad\mathbf{Q}^{\e,N}(\t)|(\Pi_{N})_{\ast}\eta\right). \label{eq:KPZ231}
\end{align}
We are left with estimating the RHS of \eqref{eq:KPZ231}. To this end, we differentiate the RHS in $\t$. As $(\Pi_{N})_{\ast}\eta$ is an invariant Gaussian measure for \eqref{eq:UVSBE1} when viewed as a multi-dimensional SDE (see Section 2 of \cite{GP16}), we get the classical entropy production inequality below (see \cite{GPV} or Appendix 1.9 in \cite{KL}, which extends to the continuous state-space case easily); the RHS below is non-positive, namely the negative of a Dirichlet form, as Markov generators are non-positive definite:
\begin{align}
\partial_{\t}\mathsf{H}\left(\grad\mathbf{Q}^{\e,N}(\t)|(\Pi_{N})_{\ast}\eta\right) \ \lesssim \ \E^{(\Pi_{N})_{\ast}\eta}\mathbf{R}^{\e,N}(\t)\mathscr{L}_{N}\mathbf{R}^{\e,N}(\t), \label{eq:KPZ232}
\end{align}
where $\mathbf{R}^{\e,N}(\t)$ is the square root of the Radon-Nikodym derivative of $\grad\mathbf{Q}^{\e,N}(\t)$ with respect to $(\Pi_{N})_{\ast}\eta$ and $\mathscr{L}_{N}$ is the generator for the SDEs \eqref{eq:UVSBE1}, again for $\mathbf{F}(\x)=\x^2$. {Let us now observe:}
\begin{itemize}
\item {As noted in \cite{GJara,GP16,GP20}, the symmetric part of the generator $\mathscr{L}_{N}$, with respect to $(\Pi_{N})_{\ast}\eta$, can be written as $\sum_{|k|\leq N}\mathbf{1}_{k\neq0}\mathscr{L}_{\mathrm{OU}}^{(k)}$, where $\mathscr{L}_{\mathrm{OU}}^{(k)}$ is the generator of a one-dimensional Ornstein-Uhlenbeck process of speed $\gtrsim1$. Indeed, \eqref{eq:UVSBE1} without the asymmetry $\mathbf{F}$ diagonalizes via Fourier transform into independent Ornstein-Uhlenbeck processes as in \cite{GJara,GP16,GP20}.}
\item {The measure $(\Pi_{N})_{\ast}\eta$ is a product measure on Fourier coordinates $k\in\Z$ such that $|k|\leq N$ and $k\neq0$, since $\eta$ is a mean-zero space-time white noise on $\mathbb{T}$. Moreover, its marginal on the $k$-th Fourier coordinate is the invariant measure for $\mathscr{L}_{\mathrm{OU}}^{(k)}$; see \cite{GJara,GP16,GP20}.}
\item {Each $\mathscr{L}_{\mathrm{OU}}^{(k)}$ satisfies a log-Sobolev inequality with respect to its invariant measure, the $k$-th marginal of the product measure $(\Pi_{N})_{\ast}\eta$, with constant $\mathrm{O}(1)$. That is, for any probability density $f$ with respect to the $k$-th marginal of $(\Pi_{N})_{\ast}\eta$, we have 
\begin{align}
\E^{(\Pi_{N})_{\ast}\eta}f\log f \ \lesssim \ -\E^{(\Pi_{N})_{\ast}\eta}\sqrt{f}\mathscr{L}^{(k)}_{\mathrm{OU}}\sqrt{f}.
\end{align}
Thus, the sum over all $\{|k|\leq N, k\neq0\}$, which gives the symmetric part of $\mathscr{L}_{N}$, satisfies a log-Sobolev inequality with constant $\mathrm{O}(1)$ with respect to the product measure $(\Pi_{N})_{\ast}\eta$ (so-called ``tensorization" of log-Sobolev inequalities). For this, see Corollary 4.2 in \cite{G}.}
\end{itemize}
This resulting LSI from the above bullet points gives us the following from \eqref{eq:KPZ232}, where $C$ is some positive universal constant that depends on nothing except the constant $1$ in front of the Laplacian in \eqref{eq:UVSBE1} and would otherwise scale with the constants in front of the Laplacian and noise therein:
\begin{align}
\partial_{\t}\mathsf{H}\left(\grad\mathbf{Q}^{\e,N}(\t)|(\Pi_{N})_{\ast}\eta\right) \ \leq \ -C\mathsf{H}\left(\grad\mathbf{Q}^{\e,N}(\t)|(\Pi_{N})_{\ast}\eta\right).
\end{align}
From the previous differential inequality, we straightforwardly deduce
\begin{align}
\mathsf{H}\left(\grad\mathbf{Q}^{\e,N}(\t)|(\Pi_{N})_{\ast}\eta\right) \ \leq \ \exp(-C\t)\cdot\mathsf{H}\left(\grad\mathbf{Q}^{\e,N}(0)|(\Pi_{N})_{\ast}\eta\right). \label{eq:KPZ233} 
\end{align}
The relative entropy factor on the RHS of \eqref{eq:KPZ233} is bounded by itself without the $N$-superscripts, again because the $N$-superscript means pushforward under $\Pi_{N}$ projections, and relative entropy is contractive with respect to pushforwards; see Lemma \ref{lemma:re2}. Combining this observation with \eqref{eq:KPZ231} completes the proof.
\end{proof}
%
%
%
\section{Fractional Stochastic Burgers Equation}
Instead of \eqref{eq:SBE}, one can study the following stochastic PDE, in which we take $\alpha\in(1/2,1]$:
\begin{align}
\partial_{\t}\mathbf{u}_{\alpha} \ = \ -(-\Delta)^{\alpha}\mathbf{u}_{\alpha} + \beta\grad(\mathbf{u}_{\alpha}^{2}) + (-\Delta)^{\alpha/2}\xi. \label{eq:fSBE}
\end{align}
Equations of this type were studied in \cite{GJara,GP20} for stationary initial data, {but for $\alpha>3/4$} well-posedness and fairly robust analytic properties of \eqref{eq:fSBE} should be derivable via regularity structures \cite{Hai14} or paracontrolled distributions \cite{GPI} ({ $\alpha=3/4$ is ``critical" for these theories}). We note, however, that the anti-derivative of \eqref{eq:fSBE} satisfies the same type of comparison principle that \eqref{eq:SPDE} does that we have used throughout this paper. Moreover, for $\alpha>3/4$, equations of the following type should converge to solutions of \eqref{eq:fSBE}, in the same sense as in Theorem \ref{theorem:KPZ1} and with the same assumptions on the nonlinearity $\mathbf{F}$ in Assumption \ref{ass:intro2} and the initial data in Assumption \ref{ass:intro3}; in the following, $\xi^{N}$ is a smoothing of the noise, for example via Fourier smoothing $\Pi_{N}$, and like in Definition \ref{definition:intro1}, we may include $\Pi_{N}$ in front of the nonlinearity $\mathbf{F}$ below as well:
\begin{align}
\partial_{\t}\mathbf{u}_{\alpha}^{N} \ = \ -(-\Delta)^{\alpha}\mathbf{u}_{\alpha}^{N} + \e_{N}^{-1}\grad\mathbf{F}(\e_{N}^{1/2}\mathbf{u}_{\alpha}^{N}) + (-\Delta)^{\alpha/2}\xi^{N}. \label{eq:fSPDE}
\end{align}
This last convergence claim can be checked for stationary white noise initial data by following the arguments of \cite{GP16}; only the assumption $\alpha>1/2$ is needed for the analysis therein, but $\alpha>3/4$ is needed for uniqueness of limit points that was established in \cite{GP20}. For the reader's convenience, let us note that the only estimate that {needs} be checked in \cite{GP16} is the following, used in the proof of the Boltzmann-Gibbs principle (Proposition 11 of \cite{GP16}):
\begin{align}
{\sum}_{1\leq|k_{1}|,|k_{2}|\leq N}(|k_{1}|^{\alpha}+|k_{2}|^{\alpha})^{-2} \ll N \quad \mathrm{where} \quad \alpha > 1/2.
\end{align}
Since we have a maximum principle for equations of the type \eqref{eq:fSPDE} (the fractional Laplacian in \eqref{eq:fSPDE} is the generator for a Markov process), we may use our methods to prove a fractional version of Theorem \ref{theorem:KPZ1} if $\alpha>3/4$, and if $\alpha>1/2$ more generally if we settle for tightness, not convergence and identification of limit points. This depends on the invariant measure of \eqref{eq:fSPDE} models to be Gaussian white noise. In the same spirit, we can also prove convergence to white noise invariant measure for arbitrary continuous initial data similar to Theorem \ref{theorem:KPZ2} but for the fractional stochastic Burgers equation \eqref{eq:fSBE}. Indeed, if we take an approximation to \eqref{eq:fSBE} by Fourier smoothing in the same way as we did for \eqref{eq:SBE} in the proof of Theorem \ref{theorem:KPZ2}, the symmetric part of the resulting finite-dimensional SDE is again given by independent Ornstein-Uhlenbeck processes with uniformly bounded from below speed. We conclude by emphasizing our methods provide a \emph{first}, intrinsic or extrinsic, notion of solutions to \eqref{eq:fSBE} with general continuous initial data; again, regularity structures \cite{Hai14} and paracontrolled distributions \cite{GPI} and energy solutions \cite{GJara,GP20} address H\"{o}lder regularity initial data, and for $\alpha\neq1$ the Cole-Hopf transform does not apply to make sense of \eqref{eq:fSBE}. {On the other hand, our method depends crucially on regularizing the fractional stochastic Burgers equation in a Markovian fashion with an explicit invariant measure, unlike regularity structures or paracontrolled distributions.}
\appendix
\section{PDE estimates}
{
\begin{lemma}\label{lemma:pde2}
For any $N>0$ and any initial data $\mathbf{u}^{N,1}(0,\cdot)\in\Pi_{N}\mathscr{L}^{2}(\mathbb{T})$, with probability 1 there exists a unique process $\mathbf{u}^{N,1}\in\mathscr{C}(\R_{\geq0},\Pi_{N}\mathscr{L}^{2}(\mathbb{T}))$ such that 
\begin{align}
\mathbf{u}^{N,1}(\t,\x) \ = \ \mathbf{u}^{N,1}(0,\x) + \int_{0}^{\t}\Delta\mathbf{u}^{N,1}(\s,\x)\d\s + \int_{0}^{\t}\e_{N}^{-1}\Pi_{N}\grad\mathbf{F}(\e_{N}^{1/2}\mathbf{u}^{N,1}(\s,\x))\d\s + \grad\Pi_{N}\mathbf{B}^{N}(\t,\x), \label{eq:pde2I}
\end{align}
where for each $\x\in\mathbb{T}$, the process $\mathbf{B}^{N}(\t,\x)$ is a continuous martingale with quadratic variation $2\t$. In particular, $\mathbf{u}^{N,1}(\t,\x)$ is smooth in $\x$ for all $\t\geq0$ with probability 1.
\end{lemma}
\begin{proof}
This follows from the argument at the beginning of Section 4 of \cite{GJara}, but $\mathbf{F}$ is not necessarily quadratic as it was in \cite{GJara}. However, all we need from $\mathbf{F}$ is that $\int_{\mathbb{T}}\mathbf{u}^{N,1}(\t,\x)\grad\mathbf{F}(\e_{N}^{1/2}\mathbf{u}^{N,1}(\t,\x)\d\x=0$ in order to use said argument from \cite{GJara}. This relation is proved in Lemma 5 of \cite{GP16}. (Technically, Lemma 5 of \cite{GP16} shows that we can move $\grad\mathbf{F}(\e_{N}^{1/2}\cdot)$ onto the other copy of $\mathbf{u}^{N,1}$ if we add a sign to the $\mathbb{T}$-integral, so that $\grad\mathbf{F}(\e_{N}^{1/2}\cdot)$ is ``asymmetric" with respect to Lebesgue measure on $\mathbb{T}$. This certainly implies$\int_{\mathbb{T}}\mathbf{u}^{N,1}(\t,\x)\grad\mathbf{F}(\e_{N}^{1/2}\mathbf{u}^{N,1}(\t,\x)\d\x=0$.) This completes the proof.
\end{proof}
\begin{lemma}\label{lemma:pde3}
Suppose $\partial_{\t}\mathbf{G}_{\t}(\x)=\Delta\mathbf{G}_{\t}(\x)$ for $\t>0$ and $\x\in\mathbb{T}$, and suppose $\mathbf{G}_{0}(\x)=\delta_{\x}$ as measures. Then for any $\Phi\in\mathscr{L}^{\infty}(\mathbb{T})$ and for any $\t\geq0$, we have
\begin{align}
|(\mathrm{e}^{\t\Delta}\Phi(\cdot))(\x)| \ &= \ |\int_{\mathbb{T}}\mathbf{G}_{\t}(\x-\y)\Phi(\y)\d\y| \ \leq \ \sup_{\y\in\mathbb{T}}|\Phi(\y)| \\
\int_{\mathbb{T}}|\grad\mathbf{G}_{\t}(\x)|\d\x \ &\lesssim \ 1+\t^{-\frac12}.
\end{align}
\end{lemma}
\begin{proof}
Let $\mathbf{H}$ solve $\partial_{\t}\mathbf{H}_{\t}(\x)=\Delta\mathbf{H}_{\t}(\x)$ and $\mathbf{H}_{0}(\x)=\delta_{\x}$ on $\R_{\geq0}\times\R$. By the method of images, we have the following representation of $\mathbf{G}$ in terms of the full-line heat kernel $\mathbf{H}$:
\begin{align}
\mathbf{G}_{\t}(\x) \ = \ \sum_{k\in\Z}\mathbf{H}_{\t}(\x+k).
\end{align}
Above, we identify $\mathbb{T}\simeq\R/\Z\simeq[0,1)$. Indeed, the RHS is $1$-periodic. It converges because $\mathbf{H}$ is Gaussian and decays sub-exponentially in space. Moreover, it is a sum of terms that vanish under $\partial_{\t}-\Delta$. Lastly, for $\x\in\mathbb{T}\simeq[0,1)$, we have $\x+k\not\in\mathbb{T}$ if $k\neq0$, so every $k\neq0$-term on the RHS vanishes at $\t=0$ and we recover just $\mathbf{H}_{0}(\x)=\delta_{\x}$. By uniqueness of solutions to the heat equation on $\mathbb{T}$ (the proof is the same as in the full line case), we prove the above relation holds. Using this relation, let us first estimate
\begin{align}
|\int_{\mathbb{T}}\mathbf{G}_{\t}(\x-\y)\Phi(\y)\d\y| \ &\leq \ \sup_{\y\in\mathbb{T}}|\Phi(\y)|\times\int_{\mathbb{T}}|\mathbf{G}_{\t}(\x-\y)|\d\y \\
&\leq \ \sup_{\y\in\mathbb{T}}|\Phi(\y)|\times\int_{\mathbb{T}}\sum_{k\in\Z}|\mathbf{H}_{\t}(\x+k)|\d\x \\
&\leq \ \sup_{\y\in\mathbb{T}}|\Phi(\y)|\int_{\R}|\mathbf{H}_{\t}(\x)|\d\x \ = \ \sup_{\y\in\mathbb{T}}|\Phi(\y)|.
\end{align}
Indeed, the last line follows by writing $\R$ as a disjoint union (over $k\in\Z$) of $k+\mathbb{T}$ and a standard Gaussian integration. We also have
\begin{align}
\int_{\mathbb{T}}|\grad\mathbf{G}_{\t}(\x)|\d\x \ &\leq \ \int_{\mathbb{T}}\sum_{k\in\Z}|\grad\mathbf{H}_{\t}(\x+k)|\d\x \ \leq \ \int_{\R}|\grad\mathbf{H}_{\t}(\x)|\d\x \ \lesssim \ 1+\t^{-\frac12},
\end{align}
where the last bound is a standard Gaussian estimate. This finishes the proof.
\end{proof}
\begin{lemma}\label{lemma:pde1}
Suppose $\mathbf{h}^{N,2}(0,\cdot)\in\mathscr{C}^{1}(\mathbb{T})$. With probability 1, there exists a unique $\mathbf{h}^{N,2}\in\mathscr{C}(\R_{\geq0},\mathscr{C}^{1}(\mathbb{T}))$ such that for all $\t\geq0$ and $\x\in\mathbb{T}$, we have 
\begin{align}
\mathbf{h}^{N,2}(\t,\x) \ &= \ (\mathrm{e}^{\t\Delta}\mathbf{h}^{N,2}(0,\cdot))(\x) \\
&+ \int_{0}^{\t}\left(\mathrm{e}^{(\t-\s)\Delta}\left(\e_{N}^{-1}\mathbf{F}(\e_{N}^{1/2}\grad(\mathbf{h}^{N,1}(\s,\cdot)+\mathbf{h}^{N,2}(\s,\cdot))) - \e_{N}^{-1}\mathbf{F}(\e_{N}^{1/2}\grad\mathbf{h}^{N,1}(\s,\cdot))\right)\right)\d\s.
\end{align}
Above, $\grad\mathbf{h}^{N,1}=\mathbf{u}^{N,1}$ from \emph{Lemma \ref{lemma:pde2}}. If, in addition to \emph{Assumption \ref{ass:intro2}}, we know $\mathbf{F}$ is smooth and satisfies $|\grad^{k}\mathbf{F}(\x)|\lesssim_{k}1$, then $\mathbf{h}^{N,2}\in\mathscr{C}(\R_{\geq0},\mathscr{C}^{2}(\mathbb{T}))$ and, with probability 1, satisfies the following PDE pointwise in $(\t,\x)$:
\begin{align}
\partial_{\t}\mathbf{h}^{N,2} \ = \ \Delta\mathbf{h}^{N,2} + \e_{N}^{-1}\mathbf{F}(\e_{N}^{1/2}\grad(\mathbf{h}^{N,1}+\mathbf{h}^{N,2})) - \e_{N}^{-1}\mathbf{F}(\e_{N}^{1/2}\grad\mathbf{h}^{N,1}). \label{eq:pde1I}
\end{align}
\end{lemma}
\begin{proof}
Let us first note that by Lemma \ref{lemma:pde2}, all derivatives of $\mathbf{h}^{N,1}$ are defined pointwise with probability 1. We now start with a contraction mapping argument. Let $\mathscr{X}$ be the set of functions $\mathbf{f}\in\mathscr{C}(\R_{\geq0},\mathscr{C}^{1}(\mathbb{T}))$ such that $\mathbf{f}(0,\cdot)=\mathbf{h}^{N,2}(0,\cdot)$. Define the nonlinear map $\mathfrak{S}$ acting on $\mathscr{X}$ by
\begin{align}
\mathfrak{S}\mathbf{f}(\t,\x) \ &= \ (\mathrm{e}^{\t\Delta}\mathbf{f}(0,\cdot))(\x) \\
&+ \int_{0}^{\t}\left(\mathrm{e}^{(\t-\s)\Delta}\left(\e_{N}^{-1}\mathbf{F}(\e_{N}^{1/2}\grad(\mathbf{h}^{N,1}(\s,\cdot)+\mathbf{f}(\s,\cdot)) - \e_{N}^{-1}\mathbf{F}(\e_{N}^{1/2}\grad\mathbf{h}^{N,1}(\s,\cdot))\right)\right)\d\s.
\end{align}
Given $\mathbf{f}_{1},\mathbf{f}_{2}\in\mathscr{X}$, because $\mathbf{f}_{1}(0,\cdot)=\mathbf{f}_{2}(0,\cdot)$ by definition of $\mathscr{X}$, we have 
\begin{align}
&\mathfrak{S}\mathbf{f}_{1}(\t,\x)-\mathfrak{S}\mathbf{f}_{2}(\t,\x) \\
&= \ \e_{N}^{-1}\int_{0}^{\t}\left(\mathrm{e}^{(\t-\s)\Delta}\left(\mathbf{F}(\e_{N}^{1/2}\grad(\mathbf{h}^{N,1}(\s,\cdot)+\mathbf{f}_{1}(\s,\cdot)))-\mathbf{F}(\e_{N}^{1/2}\grad(\mathbf{h}^{N,1}(\s,\cdot)+\mathbf{f}_{2}(\s,\cdot)))\right)\right)(\x)\d\s \label{eq:pde11a} \\
&\grad\mathfrak{S}\mathbf{f}_{1}(\t,\x)-\mathfrak{S}\mathbf{f}_{2}(\t,\x) \\
&= \ \e_{N}^{-1}\int_{0}^{\t}\left(\grad\mathrm{e}^{(\t-\s)\Delta}\left(\mathbf{F}(\e_{N}^{1/2}\grad(\mathbf{h}^{N,1}(\s,\cdot)+\mathbf{f}_{1}(\s,\cdot)))-\mathbf{F}(\e_{N}^{1/2}\grad(\mathbf{h}^{N,1}(\s,\cdot)+\mathbf{f}_{2}(\s,\cdot)))\right)\right)(\x)\d\s. \label{eq:pde11b}
\end{align}
Because $\mathbf{F}$ is uniformly Lipschitz, by Lemma \ref{lemma:pde3} and elementary manipulations, we have the following bounds in which the implied constant depends on $\mathbf{F}$ only through its Lipschitz constant:
\begin{align}
|\eqref{eq:pde11a}| \ &\lesssim_{N,\mathbf{F}} \ \int_{0}^{\t}\int_{\mathbb{T}}|\mathbf{G}_{\t-\s}(\x-\y)|\times|\grad\mathbf{f}_{1}(\s,\y)-\grad\mathbf{f}_{2}(\s,\y)|\d\y\d\s \\
&\leq \ \t\sup_{0\leq\s\leq\t}\sup_{\y\in\mathbb{T}}|\grad\mathbf{f}_{1}(\s,\y)-\grad\mathbf{f}_{2}(\s,\y)| \\
|\eqref{eq:pde11b}| \ &\lesssim_{N,\mathbf{F}} \ \int_{0}^{\t}\int_{\mathbb{T}}|\grad\mathbf{G}_{\t-\s}(\x-\y)|\times|\grad\mathbf{f}_{1}(\s,\y)-\grad\mathbf{f}_{2}(\s,\y)|\d\y\d\s \\
&\lesssim \ \int_{0}^{\t}(1+|\t-\s|^{-\frac12})\d\s\sup_{0\leq\s\leq\t}\sup_{\y\in\mathbb{T}}|\grad\mathbf{f}_{1}(\s,\y)-\grad\mathbf{f}_{2}(\s,\y)| \\
&\lesssim \ (\t+\t^{\frac12})\sup_{0\leq\s\leq\t}\sup_{\y\in\mathbb{T}}|\grad\mathbf{f}_{1}(\s,\y)-\grad\mathbf{f}_{2}(\s,\y)|.
\end{align}
By the above two displays, there exists $c=c(N,\mathbf{F})>0$ such that if $T_{0}\leq c$, then $\mathfrak{S}$ is a contraction on the space of functions $\mathbf{f}$ in $\mathscr{C}([0,T_{0}],\mathscr{C}^{1}(\mathbb{T}))$ such that $\mathbf{f}(0,\cdot)=\mathbf{h}^{N,2}(0,\cdot)$. This is a closed subspace of a Banach space $\mathscr{C}([0,T_{0}],\mathscr{C}^{1}(\mathbb{T}))$, so it itself is a Banach space. Therefore, there is a unique fixed point of $\mathfrak{S}$. This gives mild solutions until time $T_{0}$ with probability 1. Because the life-time $T_{0}$ does not depend on the initial data, we can iterate and obtain mild solutions for infinite life-time with probability 1.

It remains to show that if $\mathbf{F}$ is smooth, then we have classical $\mathscr{C}^{2}(\mathbb{T})$-valued solutions. To this end, consider the same map $\mathfrak{S}$ but acting on $\mathscr{Y}$, where $\mathscr{Y}$ is the space of functions $\mathbf{f}\in\mathscr{C}(\R_{\geq0},\mathscr{C}^{2}(\mathbb{T}))$ such that $\mathbf{f}(0,\cdot)=\mathbf{h}^{N,2}(0,\cdot)$. For any $\mathbf{f}_{1},\mathbf{f}_{2}\in\mathscr{Y}$, we have 
\begin{align}
&\grad^{2}\mathfrak{S}\mathbf{f}_{1}(\t,\x)-\grad^{2}\mathfrak{S}\mathbf{f}_{2}(\t,\x) \label{eq:pde12} \\
&= \ \e_{N}^{-1}\int_{0}^{\t}\left(\grad\mathrm{e}^{(\t-\s)\Delta}\left(\grad\mathbf{F}(\e_{N}^{1/2}\grad(\mathbf{h}^{N,1}(\s,\cdot)+\mathbf{f}_{1}(\s,\cdot)))-\grad\mathbf{F}(\e_{N}^{1/2}\grad(\mathbf{h}^{N,1}(\s,\cdot)+\mathbf{f}_{2}(\s,\cdot)))\right)\right)(\x)\d\s.
\end{align}
We now compute the $\grad\mathbf{F}$-terms. By the chain rule, for $i=1,2$ we have
\begin{align}
\grad\mathbf{F}(\e_{N}^{1/2}\grad(\mathbf{h}^{N,1}(\s,\cdot)+\mathbf{f}_{i}(\s,\cdot))) \ = \ \e_{N}^{1/2}\mathbf{F}'(\e_{N}^{1/2}\grad(\mathbf{h}^{N,1}(\s,\cdot)+\mathbf{f}_{i}(\s,\cdot)))(\grad^{2}\mathbf{h}^{N,1}(\s,\cdot)-\grad^{2}\mathbf{f}_{i}(\s,\cdot)).
\end{align}
Thus, because $\mathbf{F}$ has uniformly bounded derivatives by assumption, we deduce
\begin{align}
&|\grad\mathbf{F}(\e_{N}^{1/2}\grad(\mathbf{h}^{N,1}(\s,\cdot)+\mathbf{f}_{1}(\s,\cdot)))-\grad\mathbf{F}(\e_{N}^{1/2}\grad(\mathbf{h}^{N,1}(\s,\cdot)+\mathbf{f}_{2}(\s,\cdot)))| \\
&\lesssim_{N,\mathbf{F}} \ |\grad\mathbf{f}_{1}(\s,\cdot)-\grad\mathbf{f}_{2}(\s,\cdot)| + |\grad^{2}\mathbf{f}_{1}(\s,\cdot)-\grad^{2}\mathbf{f}_{2}(\s,\cdot)|.
\end{align}
Combining the previous three displays with Lemma \ref{lemma:pde3}, we deduce
\begin{align}
|\eqref{eq:pde12}| \ &\lesssim_{N,\mathbf{F}} \ \int_{0}^{\t}\int_{\mathbb{T}}|\grad\mathbf{G}_{\t-\s}(\x-\y)|\times\left(|\grad\mathbf{f}_{1}(\s,\y)-\grad\mathbf{f}_{2}(\s,\y)|+|\grad^{2}\mathbf{f}_{1}(\s,\y)-\grad^{2}\mathbf{f}_{2}(\s,\y)|\right)\d\y\d\s \\
&\lesssim \ \int_{0}^{\t}(1+|\t-\s|^{-\frac12})\d\s\times\sup_{0\leq\s\leq\t}\sup_{\y\in\mathbb{T}}\left(|\grad\mathbf{f}_{1}(\s,\y)-\grad\mathbf{f}_{2}(\s,\y)|+|\grad^{2}\mathbf{f}_{1}(\s,\y)-\grad^{2}\mathbf{f}_{2}(\s,\y)|\right) \\
&\lesssim \ (\t+\t^{\frac12})\sup_{0\leq\s\leq\t}\sup_{\y\in\mathbb{T}}\left(|\grad\mathbf{f}_{1}(\s,\y)-\grad\mathbf{f}_{2}(\s,\y)|+|\grad^{2}\mathbf{f}_{1}(\s,\y)-\grad^{2}\mathbf{f}_{2}(\s,\y)|\right).
\end{align}
In particular, there exists $c'=c'(N,\mathbf{F})>0$ such that if $T_{0}\leq c'$, then $\mathfrak{S}$ is a contraction on $\mathscr{Y}$, which is a closed subspace of the Banach space $\mathscr{C}(\R_{\geq0},\mathscr{C}^{2}(\mathbb{T}))$ is therefore a Banach space itself. This gives a fixed point and thus classical $\mathscr{C}^{2}(\mathbb{T})$-valued solutions for life-time $T_{0}$, which we can again bootstrap to infinite life-time, all with probability 1. To show this solution $\mathbf{h}^{N,2}$ satisfies \eqref{eq:pde1I} pointwise with probability 1, we compute by using $\partial_{\t}\mathrm{e}^{\t\Delta}=\Delta\mathrm{e}^{\t\Delta}$ and the fundamental theorem of calculus as follows:
\begin{align}
\partial_{\t}\mathbf{h}^{N,2}(\t,\x) \ &= \ \partial_{\t} (\mathrm{e}^{\t\Delta}\mathbf{h}^{N,2}(0,\cdot))(\x) \\
&+ \partial_{\t}\int_{0}^{\t}\left(\mathrm{e}^{(\t-\s)\Delta}\left(\e_{N}^{-1}\mathbf{F}(\e_{N}^{1/2}\grad(\mathbf{h}^{N,1}(\s,\cdot)+\mathbf{h}^{N,2}(\s,\cdot))) - \e_{N}^{-1}\mathbf{F}(\e_{N}^{1/2}\grad\mathbf{h}^{N,1}(\s,\cdot))\right)\right)\d\s \\
&= \ \Delta(\mathrm{e}^{\t\Delta}\mathbf{h}^{N,2}(0,\cdot))(\x) \\
&+ \ \int_{0}^{\t}\partial_{\t}\left(\mathrm{e}^{(\t-\s)\Delta}\left(\e_{N}^{-1}\mathbf{F}(\e_{N}^{1/2}\grad(\mathbf{h}^{N,1}(\s,\cdot)+\mathbf{h}^{N,2}(\s,\cdot))) - \e_{N}^{-1}\mathbf{F}(\e_{N}^{1/2}\grad\mathbf{h}^{N,1}(\s,\cdot))\right)\right)\d\s \\
&+ \ \left(\mathrm{e}^{(\t-\s)\Delta}\left(\e_{N}^{-1}\mathbf{F}(\e_{N}^{1/2}\grad(\mathbf{h}^{N,1}(\s,\cdot)+\mathbf{h}^{N,2}(\s,\cdot))) - \e_{N}^{-1}\mathbf{F}(\e_{N}^{1/2}\grad\mathbf{h}^{N,1}(\s,\cdot))\right)\right)|_{\s=\t} \label{eq:pde13}\\
&= \ \Delta(\mathrm{e}^{\t\Delta}\mathbf{h}^{N,2}(0,\cdot))(\x) \label{eq:pde14a} \\
&+ \ \Delta\int_{0}^{\t}\left(\mathrm{e}^{(\t-\s)\Delta}\left(\e_{N}^{-1}\mathbf{F}(\e_{N}^{1/2}\grad(\mathbf{h}^{N,1}(\s,\cdot)+\mathbf{h}^{N,2}(\s,\cdot))) - \e_{N}^{-1}\mathbf{F}(\e_{N}^{1/2}\grad\mathbf{h}^{N,1}(\s,\cdot))\right)\right)\d\s \label{eq:pde14b}\\
&+ \ \e_{N}^{-1}\mathbf{F}(\e_{N}^{1/2}\grad(\mathbf{h}^{N,1}+\mathbf{h}^{N,2})) - \e_{N}^{-1}\mathbf{F}(\e_{N}^{1/2}\grad\mathbf{h}^{N,1}).\label{eq:pde14c}
\end{align}
We note that setting $\s=\t$ in \eqref{eq:pde13} is, rigorously, equal to removing $\mathrm{e}^{(\t-\s)\Delta}$ (because it is a delta function at $\s=\t$) and evaluating the rest at $\s=\t$ and $\cdot=\x$. This is because what is inside $\mathrm{e}^{(\t-\s)\Delta}$ (and thus what is being integrated against an approximation to the identity) is continuous in space-time with probability 1 (because of almost sure regularity of $\mathbf{F}$ and $\grad\mathbf{h}^{N,1}$ and $\grad\mathbf{h}^{N,2}$). Now, note \eqref{eq:pde14a}-\eqref{eq:pde14b} is equal to $\Delta\mathbf{h}^{N,2}(\t,\x)$, and \eqref{eq:pde14c} is the rest of \eqref{eq:pde1I}. This completes the proof.
\end{proof}
}
%
%
%
\section{Wasserstein Distance}\label{section:w}
In this section, we specialize to the Polish space given by the space of differentiable functions on $\mathbb{T}$ that are bounded and whose derivatives are bounded. We will let $\mathscr{X}$ denote its Banach dual, equipped with the operator norm topology.
\begin{definition}
Provided any probability measures $\mathbf{P}(1),\mathbf{P}(2)$ on $\mathscr{X}$, the Wasserstein metric is given by the following in which the infimum is given over all couplings of $\mathbf{P}(1)$ and $\mathbf{P}(2)$, namely all measures $\mathbf{P}(1,2)$ on $\mathscr{X}\times\mathscr{X}$ such that the projection onto the first coordinate is $\mathbf{P}(1)$ and the projection onto the second coordinate is $\mathbf{P}(2)$:
\begin{align}
\mathsf{W}(\mathbf{P}(1),\mathbf{P}(2)) \ = \ \inf_{(\mathbf{u}^{1},\mathbf{u}^{2})\sim\mathbf{P}(1,2)}\E^{\mathbf{P}(1,2)}\|\mathbf{u}^{1}-\mathbf{u}^{2}\|_{\mathscr{X}}.
\end{align}
\end{definition}
\section{Relative Entropy}\label{section:re}
\begin{definition}\label{definition:re1}
Provided any complete separable metric space $\mathscr{S}$ and any two probability measures $\mathbf{P}(1)$ and $\mathbf{P}(2)$, the relative entropy of $\mathbf{P}(1)$ with respect to $\mathbf{P}(2)$ is equal to $\infty$ if $\mathbf{P}(1)$ is not absolutely continuous with respect to $\mathbf{P}(2)$, and if it is (so that $\mathbf{P}(1)\ll\mathbf{P}(2)$), it is the following in which $\mathbf{R}(1,2)$ is the Radon-Nikodym derivative of $\mathbf{P}(1)$ with respect to $\mathbf{P}(2)$ and in which the expectation below is with respect to $\mathbf{P}(2)$:
\begin{align}
\mathsf{H}(\mathbf{P}(1)|\mathbf{P}(2)) \ \overset{\bullet}= \ \E^{\mathbf{P}(2)}\mathbf{R}(1,2)\log\mathbf{R}(1,2).
\end{align}
\end{definition}
We now collect a set of standard relative entropy estimates that are used crucially in this paper. These estimates are used in \cite{GJS15,GP16}, for example, and are rather standard, so we only provide brief proofs. 
\begin{lemma}\label{lemma:re2}
Consider any complete separable metric space $\mathscr{S}$ with any two probability measures $\mathbf{P}^{1}$ and $\mathbf{P}^{2}$ and any event $\mathcal{E}\subseteq\mathscr{S}$. We first have the following entropy inequality:
\begin{align}
\mathbf{P}^{1}(\mathcal{E}) \ \leq \ \left(\log2+\mathsf{H}(\mathbf{P}^{1}|\mathbf{P}^{2})\right) \left(\log\left(1+\mathbf{P}^{2}(\mathcal{E})^{-1}\right)\right)^{-1}. \label{eq:re21}
\end{align}
Retain the context prior to \emph{\eqref{eq:re21}}. For another complete separable metric space $\mathscr{S}'$ and continuous map $\Pi:\mathscr{S}\to\mathscr{S}'$, if $\Pi_{\ast}$ denotes the pushforward on probability measures on $\mathscr{S}$, we have the following contraction principle:
\begin{align}
\mathsf{H}(\Pi_{\ast}\mathbf{P}^{1}|\Pi_{\ast}\mathbf{P}^{2}) \ \leq \ \mathsf{H}(\mathbf{P}^{1}|\mathbf{P}^{2}). \label{eq:re22}
\end{align}
We now let $\mathbf{P}^{1,\mathrm{dyn}}$ and $\mathbf{P}^{2,\mathrm{dyn}}$ be path-space measures on $\mathscr{C}(\R_{\geq0},\mathscr{S})$ given by the same Markov process valued in $\mathscr{S}$ but with initial law given by $\mathbf{P}^{1}$ and $\mathbf{P}^{2}$, respectively. Then the relative entropy of $\mathbf{P}^{1,\mathrm{dyn}}$ with respect to $\mathbf{P}^{2,\mathrm{dyn}}$ is controlled by that of $\mathbf{P}^{1}$ with respect to $\mathbf{P}^{2}$, so that we have the following inequality:
\begin{align}
\mathsf{H}(\mathbf{P}^{1,\mathrm{dyn}}|\mathbf{P}^{2,\mathrm{dyn}}) \ \leq \ \mathsf{H}(\mathbf{P}^{1}|\mathbf{P}^{2}). \label{eq:re23}
\end{align}
Lastly, if a sequence of path-space measures $\{\mathbf{P}^{2,\mathrm{dyn},N}\}_{N\geq0}$ on $\mathscr{C}(\R_{\geq0},\mathscr{S})$ is tight and the relative entropies of another set of path-space measures $\{\mathbf{P}^{1,\mathrm{dyn},N}\}_{N\geq0}$ with respect to $\{\mathbf{P}^{2,\mathrm{dyn},N}\}_{N\geq0}$ are uniformly bounded in $N$, then the sequence $\{\mathbf{P}^{1,\mathrm{dyn},N}\}_{N\geq0}$ on $\mathscr{C}(\R_{\geq0},\mathscr{S})$ is also tight.
\end{lemma}
\begin{proof}
The estimate \eqref{eq:re21} is (2.6) in \cite{GJS15}. {Proving \eqref{eq:re22} starts with the duality formula below for relative entropy, which itself follows by noting the convex conjugate of $\x\log\x$ is $\mathrm{e}^{\x}$ (see Corollary 4.15 in \cite{BLM}):
\begin{align}
\mathsf{H}(\Pi_{\ast}\mathbf{P}^{1}|\Pi_{\ast}\mathbf{P}^{2}) \ = \ \sup_{\psi\in\mathscr{L}^{\infty}(\mathscr{S}')}\left(\E^{\Pi_{\ast}\mathbf{P}^{1}}\psi - \log\E^{\Pi_{\ast}\mathbf{P}^{2}}\mathrm{e}^{\psi}\right). \label{eq:re221}
\end{align}
For any $\psi\in\mathscr{L}^{\infty}(\mathscr{S}')$, we have the following by definition of pushforward:
\begin{align}
\E^{\Pi_{\ast}\mathbf{P}^{1}}\psi - \log\E^{\Pi_{\ast}\mathbf{P}^{2}}\mathrm{e}^{\psi} \ = \ \E^{\mathbf{P}^{1}}\psi\circ\Pi - \log\E^{\mathbf{P}^{2}}\psi\circ\Pi.
\end{align}
If $\psi\in\mathscr{L}^{\infty}(\mathscr{S}')$, then $\psi\circ\Pi\in\mathscr{L}^{\infty}(\mathscr{S})$. Thus,
\begin{align}
\sup_{\psi\in\mathscr{L}^{\infty}(\mathscr{S}')}\left(\E^{\Pi_{\ast}\mathbf{P}^{1}}\psi - \log\E^{\Pi_{\ast}\mathbf{P}^{2}}\mathrm{e}^{\psi}\right) \ &= \ \sup_{\psi\in\mathscr{L}^{\infty}(\mathscr{S}')}\left(\E^{\mathbf{P}^{1}}\psi\circ\Pi - \log\E^{\mathbf{P}^{2}}\psi\circ\Pi\right) \\
&\leq \ \sup_{\psi\in\mathscr{L}^{\infty}(\mathscr{S})}\left(\E^{\mathbf{P}^{1}}\psi - \log\E^{\mathbf{P}^{2}}\mathrm{e}^{\psi}\right),
\end{align}
which gives \eqref{eq:re22} when combined with \eqref{eq:re221}.} The estimate \eqref{eq:re23} follows from conditioning the path-space measures $\mathbf{P}^{i,\mathrm{dyn}}$ as the same path-space measure with initial conditions sampled via $\mathbf{P}^{i}$ data, respectively, and then employing the chain rule for relative entropy (see Section 4.2 in \cite{BLM}), which cancels the contribution of the shared path-space part of each $\mathbf{P}^{i,\mathrm{dyn}}$ measure. {In particular, by said chain rule for relative entropy, we decompose $\mathsf{H}(\mathbf{P}^{1,\mathrm{dyn}}|\mathbf{P}^{2,\mathrm{dyn}})$ into relative entropy of initial data laws plus ``worst-case relative entropy" of the dynamics themselves:
\begin{align}
\mathsf{H}(\mathbf{P}^{1,\mathrm{dyn}}|\mathbf{P}^{2,\mathrm{dyn}}) \ \leq \ \mathsf{H}(\mathbf{P}^{1}|\mathbf{P}^{2})+ \sup_{s\in\mathscr{S}}\mathsf{H}(\mathbf{P}^{1,\mathrm{dyn}}|_{s}|\mathbf{P}^{2,\mathrm{dyn}}|_{s}),
\end{align}
where $\mathsf{H}(\mathbf{P}^{1,\mathrm{dyn}}|_{s}|\mathbf{P}^{2,\mathrm{dyn}}|_{s})$ is the relative entropy of the path-space measure $\mathbf{P}^{1,\mathrm{dyn}}$ conditioning on the initial data to be $s\in\mathscr{S}$, with respect to $\mathbf{P}^{2,\mathrm{dyn}}$ conditioning on the same initial data $s\in\mathscr{S}$. But these two path-space measures are the same because they have the same initial data and same dynamics, so the sup on the RHS above is equal to zero.} The last paragraph follows by \eqref{eq:re21}, which estimates complements of compact sets in $\mathscr{C}(\R_{\geq0},\mathscr{S})$ under $\mathbf{P}^{1,\mathrm{dyn},N}$ by their probabilities under $\mathbf{P}^{2,\mathrm{dyn},N}$ and using tightness of $\mathbf{P}^{2,\mathrm{dyn},N}$.
\end{proof}


%
%
%
\end{document}